\documentclass[12pt]{amsart}
\usepackage[margin=1.25in]{geometry}
\usepackage{latexsym}
\usepackage{amsmath}
\usepackage{amssymb}
\usepackage{amsthm}
\usepackage{stmaryrd}
\usepackage{amscd}
\usepackage{enumerate}
\usepackage{enumitem}
\usepackage{mathrsfs}
\usepackage[all]{xy}
\usepackage{color}
\usepackage{graphicx}
\usepackage{mathtools}
\usepackage{comment}
\usepackage{multirow}
\usepackage{url}
\usepackage[pagebackref]{hyperref}
\hypersetup{colorlinks=true}
\usepackage{aliascnt}

\makeatletter
  
  \@addtoreset{equation}{thm}
  \makeatother

\title{On the uniform positivity of $F$-signature under reduction modulo $p$}

\author{Shunsuke Takagi}
\address{Graduate School of Mathematical Sciences, University of Tokyo, 3-8-1 Komaba, Meguro-ku, Tokyo 153-8914, Japan}
\email{stakagi@ms.u-tokyo.ac.jp}

\author{Tatsuki Yamaguchi}
\address{Department of Mathematics, School of Science, Institute of Science Tokyo}
\email{yamaguchi.t.bp@m.titech.ac.jp}

\dedicatory{Dedicated to Professor Karen E.~Smith on the occasion of her sixtieth birthday.}


\def\ge{\geqslant}
\def\le{\leqslant}
\def\phi{\varphi}
\def\epsilon{\varepsilon}

\def\bar{\overline}
\def\mapsto{\longmapsto}

\def\log{\operatorname{log}}
\def\Hom{\operatorname{Hom}}
\def\Spec{\operatorname{Spec}}
\def\Proj{\operatorname{Proj}}

\def\Div{\operatorname{div}}

\def\max{\operatorname{max}}

\def\ulim{\operatorname{ulim}}
\def\Soc{\operatorname{Soc}}

\def\m{{\mathfrak m}}
\def\n{{\mathfrak n}}
\def\ba{{\mathfrak a}}
\newcommand{\F}{\mathbb{F}}
\newcommand{\N}{\mathbb{N}}
\newcommand{\Q}{\mathbb{Q}} 
\newcommand{\C}{\mathbb{C}} 
\newcommand{\R}{\mathbb{R}} 
\newcommand{\Z}{\mathbb{Z}}

\newcommand{\sO}{\mathcal{O}}

\theoremstyle{plain}
\newtheorem{thm}{Theorem}[section] 
\newtheorem{cor}[thm]{Corollary}
\newtheorem{prop}[thm]{Proposition}
\newtheorem{conj}[thm]{Conjecture}
\newtheorem*{conjAd}{Conjecture $\textup{A}_d$}
\newtheorem*{conjBd}{Conjecture $\textup{B}_d$}
\newtheorem*{conjCd}{Conjecture $\textup{C}_d$}
\newtheorem*{conjDd}{Conjecture $\textup{D}_d$}
\newtheorem*{conjEd}{Conjecture $\textup{E}_d$}

\newtheorem{lem}[thm]{Lemma}
\theoremstyle{definition} 
\newtheorem{defn}[thm]{Definition}

\newtheorem{eg}[thm]{Example} 

\theoremstyle{remark}
\newtheorem{rem}[thm]{Remark}
\newtheorem{ques}[thm]{Question}

\newtheorem{setting}[thm]{Situation}
\newtheorem{defn and notation}[thm]{Definition and Notation}
\newtheorem*{notation}{Notation} 
\newtheorem*{cl}{Claim}
\newtheorem*{clproof}{Proof of Claim}

\newtheorem*{acknowledgement}{Acknowledgments}

\begin{document}

\tolerance = 9999

\begin{abstract}
Carvajal-Rojas, Schwede and Tucker asked whether the mod $p$ reductions of a complex klt type  singularity have uniformly positive $F$-signature for all but finitely many primes $p$. 
In this paper, we give an affirmative answer to this conjecture for pure subrings of regular local rings, including reductive quotient singularities. 
We also show that the conjecture can be reduced to the Gorenstein case. 
Finally, we discuss the connection with $F$-alpha invariants for log Fano pairs, introduced by Pande as characteristic $p$ analogs of Tian's alpha invariants. 
\end{abstract}

\maketitle
\markboth{S.~TAKAGI and T.~YAMAGUCHI}{UNIFORM POSITIVITY OF $F$-SIGNATURE}

\tableofcontents

\section{Introduction}
$F$-signature is an important numerical invariant of singularities in positive characteristic. 
Suppose that $(A,\m_A)$ is a $d$-dimensional $F$-finite Noetherian local domain of characteristic $p>0$ with perfect residue field $A/\m_A$. 
Recall that an $\mathbb{F}_p$-algebra $B$ is said to be \textit{$F$-finite} if the Frobenius endomorphism $F \colon B \to B$ is finite. 
This is a mild finiteness condition that is automatically satisfied for rings essentially of finite type over a perfect field.
We fix an algebraic closure $\overline{\mathrm{Frac}(A)}$ of the fractional field $\mathrm{Frac}(A)$ of $A$, and for each integer $e \ge 1$, consider the extension ring
$A^{1/p^e}=\{x \in \overline{\mathrm{Frac}(A)} \; | \; x^{p^e} \in A\}$ of $A$. 
Let $a_e$ denote the largest number of free copies of $A$ appearing as direct summands of $A^{1/p^e}$. 
Then the limit $\lim_{e \to \infty}a_e/p^{de}$ is known to exist by Tucker \cite{Tucker12} and is called the $F$-signature $s(A)$ of $A$. 
A key feature of $F$-signature is a result of Aberbach--Leuschke \cite{AL03}, which states that $s(A)>0$ if and only if $A$ is strongly $F$-regular, a fundamental class of singularities in positive characteristic that is viewed as a characteristic $p$ counterpart of klt type singularities.  
In this paper, we use the method of reduction modulo $p$ to define $F$-signature in equal characteristic zero and investigate its properties. 

For the purposes of this introduction, consider the following ring:
\[
R=\left(\C[X_1, \dots, X_n]/(f_1, \dots, f_r)\right)_{(X_1, \dots, X_n)},
\]
where each $f_i \in \Z[X_1, \dots, X_n]$ is a polynomial with integer coefficients. 
The reduction of $R$ modulo a prime $p$, denoted by $R \bmod p$, is obtained by reducing the coefficients of the $f_i$ modulo $p$, that is, 
\[
R \bmod p=\left(\mathbb{F}_p[X_1, \dots, X_n]/(f_1 \bmod p, \dots, f_r \bmod p)\right)_{(X_1, \dots, X_n)}, 
\]
where $f_i \bmod p$ denotes the image of $f_i$ under the natural map 
\[
\Z[X_1, \dots, X_n] \to \mathbb{F}_p[X_1, \dots, X_n].
\]
The $F$-signature $s(R)$ of $R$ is then defined as 
\[s(R)=\liminf_{p \to \infty}s(R \bmod p).\] 
Since it is not known whether the limit exists except in certain specific cases (see \cite{Singh05}, \cite{CSTZ2024}, \cite{BCPT2025}), we take the limit inferior. 
More generally, for an arbitrary local ring $R$ essentially of finite type over $\mathbb{C}$, the notion of reduction modulo $p$ is defined in Definition~\ref{def:model}, and the $F$-signature $s(R)$ is defined in Definition~\ref{F-signature in char0}.  
However, in the introduction, we restrict ourselves to the above concrete special case for simplicity of exposition. 

It is natural to ask how this invariant relates to klt type singularities, because such singularities conjecturally correspond to strongly $F$-regular singularities via reduction modulo $p$:  

\begin{conj}[\textup{\cite[Question 7.2]{SS10}}]\label{klt vs SFR}
$\Spec R$ is of klt type if and only if $R \bmod p$ is strongly $F$-regular for all but finitely many $p$. 
\end{conj}
The ``only if" direction is known to hold by \cite{Takagi}. The ``if" direction also holds if $R$ is $\Q$-Gorenstein (see \cite{Hara-Watanabe}, \cite{Smith97}), or more generally, if the anti-canonical ring of $R$ is Noetherian (see \cite{CEMS18}); however, the general case remains open even in dimension three.  

In light of this (partly conjectural) correspondence and the result of Aberbach--Leuschke mentioned above, Carvajal-Rojas, Schwede and Tucker \cite{CRST} proposed the following conjecture: 
\begin{conj}[\textup{\cite[Question 5.9]{CRST}}] \label{CRST conjecture}
If $\Spec R$ is of klt type, then $s(R)>0$. 
\end{conj}

By the ``only if" direction of Conjecture \ref{klt vs SFR}, if $\Spec R$ is of klt type, then $R \bmod p$ is strongly $F$-regular, hence $s(R \bmod p)>0$. 
Nevertheless, the limit of $s(R \bmod p)$ as $p \to \infty$ may be zero, which presents a serious obstacle.  
In dimension two, Conjecture~\ref{CRST conjecture} is straightforward, as two-dimensional klt type singularities are finite quotient singularities, for which the $F$-signature is explicitly computable. 
Singh \cite{Singh05} proved Conjecture~\ref{CRST conjecture} when $R$ is a localization of a toric ring, Caminata, Shideler, Tucker and Zerman~\cite{CSTZ2024} proved it for diagonal hypersurfaces, and Brosowsky, Coskun, Pande and Tucker~\cite{BCPT2025} proved it for certain binomial hypersurfaces. However, very little is known beyond these cases. 
In this paper, we further investigate Conjecture \ref{CRST conjecture}.

Before stating the main result of this paper, we recall the definition of pure ring extensions. 
A ring extension $R \hookrightarrow S$ is said to be \textit{pure} if the natural map $M \to M \otimes_R S$ is injective for all $R$-modules $M$. 
The main result of this paper is the following:  
\begin{thm}[Theorem \ref{Uniform positivity of F-signature of pure subrings}]\label{intro main thm}
Let $(R, \m) \hookrightarrow (S, \n)$ be a pure local $\C$-algebra homomorphism between
normal local rings essentially of finite type over $\C$.  
If $\Spec S$ is of klt type and $s(S)>0$, then $s(R)>0$. 
\end{thm}

Theorem \ref{intro main thm} leads to the following applications concerning Conjecture \ref{CRST conjecture}:

\begin{cor}[Corollaries \ref{reductive quotient}, \ref{reduce to Gorenstein case}]
The following statements hold.
\begin{enumerate}[label=$(\arabic*)$]
\item Conjecture \ref{CRST conjecture} holds when $R$ is a pure subring of a regular local ring essentially of finite type over $\C$ (e.g., a reductive quotient singularity).  
\item Conjecture \ref{CRST conjecture} can be reduced to the case where $R$ is Gorenstein. 
\end{enumerate}
\end{cor}

We give a sketch of the proof of Theorem \ref{intro main thm} below. 
The first key ingredient is the notion of local $F$-alpha invariants, a local version of $F$-alpha invariants introduced by Pande \cite{Pande23}, which are  characteristic $p$ analogs of Tian's alpha invariants. 
For the ring $R$ considered above, the local $F$-alpha invariant $\alpha_F(R)$ is defined as
\[
\alpha_F(R)=\liminf_{p \to \infty}\inf_{f \in \m \bmod p}\operatorname{fpt}(R \bmod p, f)\overline{\mathrm{ord}}_{\m \bmod p}(f),
\]
where $f$ runs through all nonzero elements in the maximal ideal $\m \bmod p$ of the local ring $R \bmod p$. 
Here, $\operatorname{fpt}(R \bmod p, f)$ denotes the $F$-pure threshold of $f$, a positive characteristic analog of the log canonical threshold, and $\overline{\mathrm{ord}}_{\m \bmod p}(f)$ denotes the normalized order of $f$ with respect to $\m \bmod p$ (see Definition \ref{normalized order} for its definition). 
Following Pande's argument, one can show that $s(R)>0$ if and only if $\alpha_F(R)>0$. 
Thus, the problem is reduced to proving the positivity of the local $F$-alpha invariant. 

The second key ingredient is the notion of ultra-$F$-regularity,  originally introduced by Schoutens \cite{Schoutens05}, which is a type of singularity over $\C$ defined in terms of the purity of ultra-Frobenii. 
An \textit{ultra-Frobenius} is a ring homomorphism from $R$ constructed as the ultraproduct of Frobenius maps in characteristic $p$ over all primes $p$. 
Ultra-$F$-regularity has two main advantages. First, it descends under pure morphisms, whereas purity itself is not generally preserved under reduction modulo $p$. Second, in our setting, ultra-$F$-regularity enables us to verify the strong $F$-regularity of the pairs $(R \bmod p, f_p^{t_p})$, where $f_p \in \m \bmod p$ is a nonzero element and $t_p \ge 0$ is a real number, simultaneously for almost all $p$ (see Proposition \ref{ultra-F-regular <-> F-regular type}). 
Here, the phrase ``almost all $p$" is made precise by fixing a non-principal ultrafilter on the set of all primes (see Section \ref{ultraproduct section} for details). 
Since the $F$-pure threshold $\operatorname{fpt}(R \bmod p, f_p)$ is the supremum of the real numbers $t_p \ge 0$ such that the pair $(R \bmod p, f_p^{t_p})$ is strongly $F$-regular, one can use ultra-$F$-regularity to show that 
\[\operatorname{fpt}(R \bmod p, f_p) \ge \operatorname{fpt}(S \bmod p, f_p)\] 
for all but finitely many $p$. 
This is a crucial step in the proof of Theorem \ref{intro main thm}. 

We also provide a geometric interpretation of Conjecture \ref{CRST conjecture}. 
Schwede--Smith \cite{SS10} introduced a global analog of strongly $F$-regular pairs, called \textit{globally $F$-regular pairs}, and showed that if $(X, \Delta)$ is a log Fano pair over $\C$, then its reduction $(X \bmod p, \Delta \bmod p)$ modulo $p$ is globally $F$-regular for all but finitely many $p$. 
Pande's $F$-alpha invariant for a log Fano pair $(X, \Delta)$, denoted by 
\[
\alpha_F((X, \Delta); -(K_X+\Delta))
\] in this paper, is formulated in terms of \textit{global $F$-split thresholds} of the reductions $(X \bmod p, \Delta \bmod p)$. 
Given a globally $F$-regular pair $(V, B)$ and a $\Q$-Weil divisor $D$ on $V$, 
the global $F$-split threshold $\mathrm{gfst}((V,B);D)$ of $D$ with respect to $(V, B)$ is defined as 
\[
\mathrm{gfst}((V,B);D):=\sup\{t \ge 0 \; | \; (V, B+tD) \; \textup{is globally $F$-regular}\}.
\]
By considering Koll\'ar components, we have the following theorem: 
\begin{thm}[Theorem \ref{local-global thm}]
If $\alpha_F((X, \Delta); -(K_X+\Delta))>0$ for all projective log Fano pairs $(X,\Delta)$ over $\C$, then Conjecture \ref{CRST conjecture} holds. 
\end{thm}

\begin{acknowledgement}
The authors are indebted to Kenta Sato for helpful comments and for sharing the proof of Lemma \ref{normalized order under pure mophisms}. 
They are also grateful to Yuji Odaka and Suchitra Pande for valuable discussions and to the referees for their careful reading of the manuscript and many useful comments. 
The first author was supported by JSPS KAKENHI Grant Numbers 20H00111, 23K22383 and 25H00399, and the second author by Grant Number 24KJ1040. 
This material is based upon work supported by the National Science Foundation under Grant No.~DMS-1928930 and by the Alfred P. Sloan Foundation under grant~G-2021-16778, while the first author was in residence at the Simons Laufer Mathematical Sciences Institute (formerly MSRI) in Berkeley, California, during the Spring 2024 semester. 
Additional support was provided by the Research Institute for Mathematical Sciences, an International Joint Usage/Research Center located in Kyoto University. 
\end{acknowledgement}

\section{Preliminaries}

\subsection{Birational geometry}
In this subsection, we briefly review some basic terminology from birational geometry. 

Let $X$ be a normal variety over a field $k$ of characteristic zero. 
Given a $\Q$-Weil divisor $D$ on $X$, the graded ring $\bigoplus_{i \ge 0} H^0(X, \sO_X(\lfloor iD \rfloor))$ is called the \textit{section ring} of $D$ and is denoted by $R(X, D)$. 

\begin{defn}
\begin{enumerate}[label=(\roman*)]
\item 
Let $\Delta$ be an effective $\Q$-Weil divisor on $X$ such that $K_X+\Delta$ is $\Q$-Cartier. 
Take a log resolution $\pi:Y \to X$ of the pair $(X, \Delta)$. 
Then $(X, \Delta)$ is said to be \textit{klt} if $\mathrm{ord}_E(K_Y-\pi^*(K_X+\Delta))>-1$ for all prime divisors $E$ on $Y$. This definition is independent of the choice of the log resolution $\pi$. 
\item $X$ is said to be of \textit{klt type} if there exists an effective $\Q$-Weil divisor $\Delta$ on $X$ such that $K_X+\Delta$ is $\Q$-Cartier and $(X, \Delta)$ is klt. 
\item $X$ is said to have only \textit{klt singularities} if $-K_X$ is $\Q$-Cartier and $(X,0)$ is klt. 
\end{enumerate}
\end{defn}

\begin{defn}
Suppose that $X$ is projective over $k$. 
\begin{enumerate}[label=(\roman*)]
\item 
Let $\Delta$ be an effective $\Q$-Weil divisor on $X$ such that $K_X+\Delta$ is $\Q$-Cartier. 
The pair $(X, \Delta)$ is said to be \textit{log Fano} if $-(K_X+\Delta)$ is ample and $(X, \Delta)$ is klt. 
\item 
$X$ is said to be of \textit{Fano type} if there exists an effective $\Q$-Weil divisor $\Delta$ on $X$ such that $K_X+\Delta$ is $\Q$-Cartier and $(X, \Delta)$ is log Fano. 
\item 
We say that $X$ is \textit{$\Q$-Fano} if $-K_X$ is ample $\Q$-Cartier and $X$ has only klt singularities. 
\end{enumerate}
\end{defn}

\subsection{\texorpdfstring{$F$-singularities}{F-singularities}}
In this subsection, we recall basic notions concerning $F$-singularities. The reader is referred to \cite{Takagi-Watanabe} and \cite{SS10} for a general overview. 

Given a module $M$ over a ring $R$ of prime characteristic $p$, let $F^e_*M$ denote the $R$-module corresponding to the pushforward $F^e_*\widetilde{M}$ of the quasi-coherent sheaf $\widetilde{M}$ on $X:=\Spec R$ corresponding to $M$ via the $e$-times iterated Frobenius morphism $F^e:X \to X$. 
Moreover, if $R$ is a Noetherian normal domain and $D$ is a Weil divisor on $X=\Spec R$, then the $R$-module $H^0(X, \sO_X(D))$ is denoted by $R(D)$. 

First we recall the definition of test ideals and strongly $F$-regular singularities. 
\begin{defn}[{cf.~\cite{Hara-Yoshida}, \cite{Takagi}}]
Suppose that $R$ is a Noetherian normal domain of characteristic $p>0$ and $\Delta$ is an effective $\Q$-Weil divisor on $\Spec R$. 
We assume that $R$ is \textit{$F$-finite}, that is, the Frobenius map $F:R\to R$ is finite. 
Let $\ba$ be a nonzero ideal of $R$ and let $t \ge 0$ be a real number.
\begin{enumerate}[label=(\roman*)]
\item For every $R$-module $M$, the submodule $0_{M}^{*(\Delta,\ba^t)}$ of $M$ consists of all elements $z\in M$ for which there exists a nonzero element $c\in R$ such that for all large integers $e$, we have 
\[
F^e_*(c\ba^{\lceil tp^e \rceil})\otimes z=0 \in F^e_*R(\lceil p^e\Delta \rceil)\otimes_R M.
\]
\item The \textit{test ideal} $\tau(R,\Delta,\ba^t)$ associated to the triple $(R,\Delta,\ba^t)$ is defined as
\[
\tau(R,\Delta,\ba^t)=\bigcap_{M} \operatorname{Ann}_R 0_{M}^{*(\Delta,\ba^t)},
\]
where $M$ runs through all $R$-modules. 
When $\Delta=0$, the ideal $\tau(R,\Delta,\ba^t)$ is denoted by $\tau(R,\ba^t)$. 
Similarly, when $\ba=R$, the ideal $\tau(R,\Delta,\ba^t)$ is denoted by $\tau(R,\Delta)$. If both $\Delta=0$ and $\ba=R$, then $\tau(R,\Delta,\ba^t)$ is simply denoted by $\tau(R)$.
\item The triple $(R,\Delta,\ba^t)$ is said to be \textit{strongly $F$-regular} if $\tau(R,\Delta,\ba^t)=R$. The strong $F$-regularity of $(R,\Delta)$, $(R,\ba^t)$ and $R$ is defined similarly. 
Throughout this paper, all strongly $F$-regular rings are assumed to be Noetherian and $F$-finite. 
\end{enumerate}
\end{defn}

\begin{rem}\label{test ideal remark}
If $(R,\m)$ is a local ring, then $\tau(R,\Delta,\ba^t)=\operatorname{Ann}_R 0_{E}^{*(\Delta,\ba^t)}$, where $E$ is the injective hull of the residue field $R/\m$ (see \cite[Definition-Proposition 3.3]{BSTZ}; cf.\ \cite[Proposition 8.23]{Hochster-Huneke90}). 
The formation of test ideals commutes with localization and completion (see \cite[Remark 3.6]{BSTZ}; cf.\ \cite[Propositions 3.1 and 3.2]{Hara-Takagi}). 
\end{rem}

\begin{defn}[\cite{Schwede08}, \cite{TW04}]
Let $R$ be a Noetherian $F$-finite integral domain of characteristic $p>0$, $\ba$ be a nonzero ideal of $R$ and $t \ge 0$ be a real number. 
\begin{enumerate}
\item 
We say that the pair $(R, \ba^t)$ is \textit{sharply $F$-pure} if there exist an integer $e \ge 1$ and a nonzero element $f\in \ba^{\lceil t(p^e-1)\rceil}$ such that the composition 
\[
R \to F^e_*R \xrightarrow{\cdot F^e_*f} F^e_*R
\]
of the $e$-times iterated Frobenius map $R \to F^e_*R$ and multiplication by $F^e_*f$ splits as an $R$-module homomorphism. 
When $\ba=R$, we simply say that $R$ is $F$-pure. 
\item 
Assume that $\ba\subsetneq R$. Then the \textit{$F$-pure threshold} $\operatorname{fpt}(\ba)$ of $\ba$ is defined by  
\[
\operatorname{fpt}(\ba)=\sup\{t\in \R_{\ge0} \mid  \text{$(R,\ba^t)$ is sharply $F$-pure}\}.
\]
If $R$ is not $F$-pure, then $\operatorname{fpt}(\ba)$ is defined to be zero for every $\ba$. 
When it is necessary to specify the base ring, $\operatorname{fpt}(R,\ba)$ is written for $\operatorname{fpt}(\ba)$. 
If $\ba=(f)$, we write $\operatorname{fpt}(f)$ for $\operatorname{fpt}(\ba)$.
\end{enumerate}
\end{defn}

Next, we introduce a global version of strong $F$-regularity and $F$-pure thresholds. 
\begin{defn}
Let $X$ be a Noetherian $F$-finite normal integral scheme of characteristic $p>0$ and let $\Delta$ be an effective $\Q$-Weil divisor on $X$. 
\begin{enumerate}[label=(\roman*)]
\item (\cite{SS10}) The pair $(X,\Delta)$ is said to be \textit{globally $F$-regular} if for every effective Weil divisor $D$ on $X$, there exists an integer $e \ge 1$ such that the composite map 
\[
\sO_X \to F^e_*\sO_X \hookrightarrow F^e_*\sO_X(\lceil (p^e-1)\Delta +D\rceil)
\]
splits as an $\sO_X$-module homomorphism, where the first map is the $e$-times iterated Frobenius map $F^e$ and the second is the pushforward of the natural inclusion $\sO_X \hookrightarrow \sO_X(\lceil (p^e-1)\Delta +D\rceil)$ by $F^e$. 
\item Suppose that $(X,\Delta)$ is globally $F$-regular. 
The \textit{global $F$-split threshold} $\mathrm{gfst}((X,\Delta);D)$ of an effective $\Q$-Weil divisor $D$ on $X$ with respect to the pair $(X, \Delta)$ is defined as
\[
\mathrm{gfst}((X,\Delta);D)=\sup\{t\in \mathbb{R}_{\ge 0} \mid \text{$(X,\Delta+tD)$ is globally $F$-regular}\}.
\]
When $\Delta=0$, we simply write $\mathrm{gfst}(X,D)$. 
If $(X,\Delta)$ is not globally $F$-regular, then we set $\mathrm{gfst}((X,\Delta);D)=0$ for any $D$. 
\end{enumerate}
\end{defn}

\begin{rem}\label{fpt vs gfst}
$F$-pure threshold can be defined for pairs $(X, D)$, where $X$ is a projective variety over an $F$-finite field and $D$ is a Cartier divisor on $X$. 
However, the global $F$-split threshold is, in general, strictly smaller than the $F$-pure threshold. 
For instance, if $k$ is a perfect field of characteristic two, $X = \mathbb{P}_k^3$, and $D \subseteq X$ is the hypersurface defined by $x^3 + y^3 + z^3 + w^3 = 0$, then $\operatorname{fpt}(X, D) = 1$ since $D$ is smooth, while $\mathrm{gfst}(X, D) = 1/2$.
\end{rem}

We now turn to the definition of the $F$-signature, the central object of study in this paper. 
\begin{defn}[\cite{HL02} cf.~\cite{Smith-vandenBergh}]\label{F-signature def}
Let $(R, \m, k)$ be a $d$-dimensional Noetherian $F$-finite local ring of characteristic $p>0$. 
For each integer $e \ge 1$, let $a_e$ denote the largest number of free copies of $R$ appearing as direct summands of $F^e_*R$, that is, $F^e_*R \cong R^{a_e} \oplus M_e$ where $M_e$ is an $R$-module with no free direct summands. 
Then $F$-signature $s(R)$ of $R$ is defined as 
\[
s(R)=\lim_{e \to \infty}\frac{a_e}{p^{e(d+\log_p[k:k^p])}}.
\]
This limit is known to exist by \cite{Tucker12}. 
\end{defn}

\begin{prop}[\cite{AL03}]\label{F-signature basic}
Let the notation be the same as in Definition \ref{F-signature def}. 
\begin{enumerate}[label=$(\arabic*)$]
\item $0 \le s(R) \le 1$. 
\item $s(R)=1$ if and only if $R$ is regular. 
\item $s(R)>0$ if and only if $R$ is strongly $F$-regular. 
\end{enumerate}
\end{prop}

In order to define $F$-signature in equal characteristic zero in Section \ref{F-signature char 0 section}, we recall the method of reduction modulo $p$. 
We briefly outline the construction below.
\begin{defn}\label{def:model}
Let $x\in X$ be a (not necessarily closed) point of a variety $X$ over $\C$. 
We say that $(A,x_A\in X_A)$ is a \textit{model} of $x\in X$ if the following conditions hold:
\begin{enumerate}[label=(\roman*)]
\item $A$ is a finitely generated $\Z$-subalgebra of $\C$.
\item $X_A$ is a scheme of finite type over $A$ and $x_A$ is a point of $X_A$ such that there exists a morphism $X \to X_A$ sending $x$ to $x_A$ making the following diagram commute:
\[
\xymatrix{
\overline{\{x\}} \ar@{^{(}->}[r] \ar[d] \ar@{}[dr]|\square & X \ar[r] \ar[d] \ar@{}[dr]|\square& \Spec \C \ar[d]\\
\overline{\{x_A\}} \ar@{^{(}->}[r] & X_A \ar[r]& \Spec A, 
}
\]
\end{enumerate}
where $\overline{\{x\}}$ (resp.~$\overline{\{x_A\}}$) denotes the closure of $\{x\}$ in $X$ (resp.~of $\{x_A\}$ in $X_A$) with the reduced induced scheme structure.

By generic freeness, after replacing $A$ by a localization $A[1/a]$ at some element $a \in A$, we may assume that $X_A$ and $\overline{\{x_A\}}$ are flat over $A$. 
For every closed point $\mu \in \Spec A$, we define $X_{\mu}$ as the scheme-theoretic fiber of $X_A \to \Spec A$ over $\mu$, that is, $X_{\mu}=X_A\times_{\Spec A}\Spec A/\mu$. 
We refer to $X_{\mu}$ as a \textit{reduction modulo $p$} of $X$. 
For a general closed point $\mu$, 
the fiber $\overline{\{x_A\}}\times_{\Spec A}\Spec A/\mu$ of $\overline{\{x_A\}} \to \Spec A$ over $\mu$ is an integral closed subscheme of $X_\mu$, and we define $x_\mu$ as its generic point. 
When $X$ is affine and $(R,\m)=\sO_{X,x}$, we define $R_A:=\sO_{X_A}$ and  $R_\mu:=\sO_{X_\mu,x_\mu}$, and write $\m_A$ for the prime ideal of $R_A$ corresponding to $x_A$. We refer to $(A, R_A, \m_A)$ as a \textit{model} of $R$. 

If $X$ is normal, then $X_{\mu}$ is a normal integral scheme for general closed points $\mu \in \Spec A$ by \cite[Th\'eor\`eme 12.2.4]{EGAIV3}. 
Then, for a $\mathbb{Q}$-Weil divisor $\Delta$ on $X$, an argument similar to the one above yields a reduction $\Delta_{\mu}$ of $\Delta$, which is a $\mathbb{Q}$-Weil divisor on $X_{\mu}$. 
The reader is referred to \cite[\S\S 11--12]{EGAIV3} and \cite{HH_char0} for basic properties of reduction modulo $p$.

\end{defn}
 
There is a significant correspondence between $F$-singularities and singularities in complex birational geometry via reduction modulo $p$. 
 \begin{thm}\label{klt SFR correspondence}
 Let $X$ be a normal variety over $\C$ and let $\Delta$ be an effective $\Q$-Weil divisor on $X$ such that $K_X+\Delta$ is $\Q$-Cartier. Fix a model $(A,X_A,\Delta_A)$ of $(X,\Delta)$.
 \begin{enumerate}[label=$(\arabic*)$]
 \item $($\cite{Hara-Watanabe}, \cite{Takagi}$)$ $(X,\Delta)$ is klt if and only if $(X_\mu,\Delta_\mu)$ is strongly $F$-regular for general closed points $\mu\in \Spec A$.
 \item $($\cite{SS10}$)$ 
Suppose that $X$ is projective. 
If $(X, \Delta)$ is a log Fano pair, then $(X_\mu,\Delta_\mu)$ is globally $F$-regular for general closed points $\mu\in \Spec A$.
\end{enumerate}
\end{thm}

We will use Lemma \ref{graded Frational} later in the proof of one of the main results of this paper. 
Although this lemma may be well known to experts, we include a proof here due to the lack of a suitable reference. Before stating it, we recall the definition of parameter test submodules and $F$-rational singularities. 

\begin{defn}[cf.~\textup{\cite{Schwede-Takagi08}}]
Suppose that $R$ is an $F$-finite integral domain with canonical module $\omega_R$. 
Let $\ba \subseteq R$ be a nonzero ideal and $t \ge 0$ a real number. 
\begin{enumerate}[label=(\roman*)]
\item 
The \textit{parameter test submodule} $\tau(\omega_R, \ba^t)$ of the pair $(R, \ba^t)$ is defined as the unique smallest nonzero submodule $M$ of $\omega_R$ satisfying that $\mathrm{Tr}_{F^e}(F^e_*\ba^{\lceil t(p^e-1)\rceil}M) \subseteq M$ for every integer $e \ge 1$, where $\mathrm{Tr}_{F^e}:F^e_*\omega_R \to \omega_R$ is the trace map of the $e$-times iterated Frobenius map $F^e$. 
\item 
The pair $(R, \ba^t)$ is said to be \textit{$F$-rational} if $R$ is Cohen--Macaulay and $\tau(\omega_R, \ba^t)=\omega_R$. 
\end{enumerate}
\end{defn}

\begin{rem}\label{F-rational rem}
As observed in \cite[Proposition 6.5 (1)]{Schwede-Takagi08}, if $(R, \ba^t)$ is strongly $F$-regular, then it is $F$-rational. Moreover, when $R$ is Gorenstein, the converse also holds. 
\end{rem}

\begin{lem}\label{graded Frational}
Let $R=\bigoplus_{n \in \Z}R_n$ be an $F$-finite $\Z$-graded domain with unique homogeneous maximal ideal $\m$. 
For every nonzero homogeneous ideal $\ba$ of $R$ and every real number $t\ge 0$, the pair $(R, \ba^t)$ is $F$-rational if and only if so is the localization $(R_{\m}, (\ba R_{\m})^t)$ at $\m$. 
\end{lem}
\begin{proof}
Let $A$ be an $F$-finite Cohen--Macaulay local (resp.~${}^*$local) domain with canonical (resp.~${}^*$canonical) module $\omega_A$, let $I \subseteq A$ be a nonzero ideal and $s\ge 0$ a real number.
The parameter test ideal $\tau_{\rm par}(A, I^s)$ is defined as 
\[
\tau_{\rm par}(A, I^s)=\tau(\omega_A, I^s):_A \omega_A \subseteq A, 
\]
where $\tau(\omega_A, I^s)$ is the parameter test submodule of the pair $(A, I^s)$. 
Note that $\tau_{\rm par}(A, I^s)=A$ if and only if $(A, I^s)$ is $F$-rational. 

It is well known that $R$ is Cohen--Macaulay if and only if so is $R_{\m}$. Therefore, we may assume that $R$ is Cohen--Macaulay. 
Pick a homogeneous test element $c \in R$. 
Then, by an argument similar to the proof of \cite[Lemma 2.1]{Hara-Takagi},  $\tau(\omega_R, \ba^t)$ can be written as 
\[
\tau(\omega_R, \ba^t)=\sum_{e \in \N}\sum_{\phi^{(e)}}\phi^{(e)}(F^e_*(c \ba^{\lceil t p^e \rceil}\omega_R)) \subseteq \omega_R
\]
where $\phi^{(e)}$ runs through all elements of $\Hom_R(F^e_*\omega_R, \omega_R)$. 
Considering the homogeneous decomposition of $\phi^{(e)}$, we may assume that each $\phi^{(e)}$ is a homogeneous element. 
This description shows $\tau(\omega_R, \ba^t)$ is a homogeneous submodule of $\omega_R$, and its formation commutes with localization. 
The same holds for $\tau_{\rm par}(R, \ba^t)$. 
Thus, it follows that $(R, \ba^t)$ is not $F$-rational if and only if $\tau_{\rm par}(R, \ba^t) \subseteq \m$, which is equivalent to saying that $\tau_{\rm par}(R_{\m}, (\ba R_{\m})^t) \subseteq \m R_{\m}$. 
This last condition means that the localized pair $(R_{\m}, (\ba R_{\m})^t)$ is not $F$-rational. 
\end{proof}

\subsection{Ultraproducts}\label{ultraproduct section}
In this subsection, we quickly review basic notions on ultraproducts and their applications to commutative algebra. 
We refer the reader to \cite{Schoutens03}, \cite{Schoutens08} and \cite[Section 3]{Yamaguchi25a} for details.
Throughout this subsection, let $W$ be an infinite set, $R$ a local ring essentially of finite type over $\C$ and let $\mathcal{P}$ denote the set of all prime numbers. 

\begin{defn}
A non-empty subset $\mathcal{F}$ of the power set of $W$ is said to be a {\it non-principal ultrafilter} on $W$ if $\mathcal{F}$ satisfies the following four conditions:
\begin{enumerate}[label=(\roman*)]
		\item If $A\in \mathcal{F}$ and $A\subseteq B\subseteq W$, then $B\in \mathcal{F}$.
		\item If $A, B\in \mathcal{F}$, then $A \cap B \in \mathcal{F}$.
		\item If $A$ is a finite subset of $W$, then $A \notin \mathcal{F}$.
		\item If $A\subseteq W$, then $A\in \mathcal{F}$ or $W\setminus A\in \mathcal{F}$.
\end{enumerate}
\end{defn}

\begin{rem}
Let $A$ be an infinite subset of $W$. By Zorn's Lemma, there exists a non-principal ultrafilter $\mathcal{F}$ on $W$ such that $A\in \mathcal{F}$.
\end{rem}

\begin{defn}
Let $\mathcal{F}$ be a non-principal ultrafilter on $W$ and let $\phi$ be a property concerning elements of $W$. 
We say that $\phi$ holds \textit{for almost all $w$} if 
\[
\{w\in W \; | \; \text{$\phi(w)$ holds}\} \in \mathcal{F}.
\]
\end{defn}

We now introduce the notion of ultraproducts, a fundamental construction in model theory and non-standard analysis.
\begin{defn}
Let $(A_w)_{w\in W}$ be a family of non-empty sets indexed by $W$ and let $\mathcal{F}$ be a non-principal ultrafilter on $W$. 
The \textit{ultraproduct} $A_\infty$ of $(A_w)_{w\in W}$ is defined as	
\[
		A_\infty=\ulim_w A_w:=\prod_{w}A_w/\sim,
	\]
where $\sim$ is the equivalence relation on $\prod_w A_w$ defined for any elements $(a_w), (b_w)\in \prod_w A_w$ by $(a_w)\sim (b_w)$ if and only if $a_w=b_w$ for almost all $w$. 
For each $(a_w)\in \prod_w A_w$, the equivalence class of $(a_w)$ in $A_\infty$ is denoted by $\ulim_w a_w$.

If $A_w=\N$ (resp. $A_w=\Q$) for all $w\in W$, then we write ${^*\N}$ (resp.~${^*\Q}$) to denote $\ulim_w A_w$. 
An element of ${^*\N}$ (resp.~${^*\Q}$) is called a \textit{non-standard natural number} (resp.~\textit{non-standard rational number}).
\end{defn}

For a family $(f_w:A_w\to B_w)_{w\in W}$ of maps, we can construct the ultraproduct $f_\infty:A_\infty \to B_\infty$ in a natural way. Similarly, we can also define the ultraproduct of binary operators.

\begin{rem}
If $(A_w)_{w \in W}$ is a family of rings, then $A_\infty$ is naturally equipped with a ring structure. 
Moreover, if $(M_w)_{w\in W}$ is a family of $A_w$-modules, then $M_\infty:=\ulim_w M_w$ carries a natural $A_\infty$-module structure.
\end{rem}

The following notation will be used throughout this paper. 
\begin{notation}
Let $\mathcal{F}$ be a non-principal ultrafilter on $\mathcal{P}$.
\begin{enumerate}[label=(\roman*)]
\item The non-standard natural number $\ulim_p p$ is denoted by $\pi$. 
\item Let $\nu=\ulim_p n_p$ and $\epsilon=\ulim_p e_p$ be non-standard natural numbers. 
Then $\nu^{\epsilon}$ is defined to be $\ulim_p n_p^{e_p}$.
\item Let $\tau=\ulim_p t_p$ be a non-standard rational number. 
Then $\lfloor \tau \rfloor:=\ulim_p \lfloor t_p\rfloor$ and $\lceil \tau \rceil:=\ulim_p \lceil t_p \rceil$. 
\end{enumerate}
\end{notation}

The following is a key fact that enables us to apply the theory of ultraproducts to commutative algebra. 
\begin{prop}
Let $\mathcal{F}$ be a non-principal ultrafilter on $\mathcal{P}$ and let $\overline{\mathbb{F}_p}$ be an algebraic closure of the finite field $\mathbb{F}_p$. 
Then there exists a (non-canonical) field isomorphism
	\[
		\ulim_p \overline{\F_p} \cong \C.
	\]
\end{prop}

In the following, we fix a non-principal ultrafilter $\mathcal{F}$ on $\mathcal{P}$ and a field  isomorphism $\ulim_p \overline{\mathbb{F}_p}\cong \C$. 
Let $\C[X_1,\dots,X_n]$ (resp.~$\overline{\mathbb{F}_p}[X_1,\dots,X_n]$) denote the polynomial ring in $n$ variables over $\C$ (resp.~$\overline{\mathbb{F}_p}$). 

\begin{thm}[\cite{vdD79}]
The natural ring homomorphism  
\[
\C[X_1,\dots,X_n]\to \ulim_p(\overline{\mathbb{F}_p}[X_1,\dots,X_n]); \quad X_i \mapsto \ulim_p X_i,
\] 
which depends on the choice of the ultrafilter $\mathcal{F}$ and the field isomorphism $\ulim_p \overline{\mathbb{F}_p}\cong \C$, is faithfully flat.
\end{thm}

\begin{defn}[{\cite[Section 3.2]{Schoutens03}}]
\begin{enumerate}[label=(\roman*)]
\item Let $f\in \C[X_1,\dots,X_n]$. 
A family of polynomials $(f_p)_{p\in \mathcal{P}}\in \prod_p (\overline{\mathbb{F}_p}[X_1,\dots,X_n])$ is said to be an \textit{approximation} of $f$ if $f$ is mapped to $\ulim_p f_p$ under the natural homomorphism $\C[X_1,\dots,X_n] \to \ulim_p(\overline{\mathbb{F}_p}[X_1,\dots,X_n])$. 
\item Let $I=(f_1,\dots,f_m)$ be an ideal of $\C[X_1,\dots,X_n]$. A family of ideals $(I_p)_{p\in \mathcal{P}}$ is said to be an $\textit{approximation}$ of $I$ if $I_p=(f_{1,p},\dots,f_{m,p})\subseteq \overline{\mathbb{F}_p}[X_1,\dots,X_n]$ for almost all $p$, where $(f_{i,p})_{p \in \mathcal{P}}$ is an approximation of $f_i$ for each $i=1, \dots, m$. 
Note that the definition of approximations of $I$ is independent of the choice of generators $f_1, \dots, f_m$ and of their approximations $(f_{i,p})_{p \in \mathcal{P}}$. 
More precisely, if $(I_p)_{p\in \mathcal{P}}$ and $(J_p)_{p\in \mathcal{P}}$ are two approximations of $I$, then $I_p=J_p$ for almost all $p$.
\end{enumerate}
\end{defn}

\begin{prop}[\cite{vdD79}]
Let $I=(f_1,\dots,f_m)$ be a prime ideal of $\C[X_1,\dots,X_n]$ and let $(I_p)_{p\in\mathcal{P}}$ be an approximation of $I$. Then $I(\ulim_p(\overline{\mathbb{F}_p}[X_1,\dots,X_n]))$ is a prime ideal of $\ulim_p(\overline{\mathbb{F}_p}[X_1,\dots,X_n])$, and $I_p$ is prime for almost all $p$.
\end{prop}

As an application of this proposition, we can construct approximations for local rings essentially of finite type over $\C$ as follows.

\begin{defn}[{\cite[Section 4.3]{Schoutens03}}]
 Let $R$ be a local ring essentially of finite type over $\C$, that is, $R \cong (\C[X_1,\dots,X_n]/I)_\mathfrak{p}$, where $I \subseteq \C[X_1, \dots, X_n]$ is an ideal and $\mathfrak{p} \subseteq \C[X_1, \dots, X_n]$ is a prime ideal containing $I$. 
 Let $(I_p)_{p\in \mathcal{P}}$ be an approximation of $I$ and let $(\mathfrak{p}_p)_{p\in\mathcal{P}}$ be an approximation of $\mathfrak{p}$. 
 A family $(R_p)_{p\in\mathcal{P}}$ of rings is said to be an \textit{approximation} of $R$ if $R_p=(\overline{\mathbb{F}_p}[X_1,\dots,X_n]/I_p)_{\mathfrak{p}_p}$ for almost all $p$,  and the ultraproduct $R_\infty:=\ulim_p R_p$ is called the \textit{non-standard hull} of $R$. 
The non-standard hull is independent of the choice of presentation of $R$, i.e., any two non-standard hulls are isomorphic as $R$-algebras.
\end{defn}

\begin{rem}
  The natural ring homomorphism $R\to R_\infty$ is known to be flat and local.
\end{rem}

\begin{rem}\label{approximation vs reduction}
 Approximations of $R$ are closely related to its reduction modulo $p$. Specifically, the following holds (\cite[Corollary 5.9]{Schoutens04}):
 Let $x\in X$ be a point of a variety $X$ over $\C$ and let $(A, x_A\in X_A)$ be a model of $x\in X$. 
 Suppose that $(R_p)_{p\in\mathcal{P}}$ is an approximation of $R:=\sO_{X,x}$.  
 For a general closed point $\mu \in \Spec A$, the base change $\overline{\{x_\mu\}}\times_{\Spec A/\mu} \Spec \overline{A/\mu}$ is an integral closed subscheme of 
 \[
 X_{\overline{\mu}}:=X_\mu \times_{\Spec A/\mu}\Spec \overline{A/\mu},
 \]
 where $\overline{A/\mu}$ is an algebraic closure of the finite field $A/\mu$, and we define $x_{\overline{\mu}}$ as its generic point. Then there exists a family of closed points $\{\mu_p\}_{p\in\mathcal{P}}$ of $\Spec A$ such that $R_p\cong \sO_{X_{\overline{\mu_p}},x_{\overline{\mu_p}}}$ for almost all $p$.
\end{rem}

\begin{defn}
\begin{enumerate}
\item Let $f\in R$. A family $(f_p)_{p\in \mathcal{P}}$ of elements is said to be an \textit{approximation} of $f$ if $f_p\in R_p$ for all $p \in \mathcal{P}$ and $f$ is mapped to $\ulim_p f_p$ under the natural map $R\to R_\infty$.
\item Let $I=(f_1,\dots,f_m)$ be an ideal of $R$, and suppose that $(f_{i,p})_{p\in \mathcal{P}}$ is an approximation of $f_i$ for each $i=1, \dots, m$. 
A family $(I_p)_{p\in \mathcal{P}}$ of ideals is said to be an \textit{approximation} of $I$ if $I_p=(f_{1,p},\dots,f_{m,p})$ for almost all $p$.
\end{enumerate}
\end{defn}

Basic ring-theoretic properties of $R$ are retained by its approximations. 
\begin{prop}[{\cite[Theorems 4.5 and 4.6]{Schoutens03}, \cite[Proposition 3.9]{Yam23a}}]
Let $(R_p)_{p\in\mathcal{P}}$ be an approximation of $R$. 
\begin{enumerate}[label=$(\arabic*)$]
\item Let $d$ be a natural number. Then $R$ is of dimension $d$ if and only if so is $R_p$ for almost all $p$.
\item $R$ is regular (resp. Gorenstein, Cohen--Macaulay, normal) if and only if so is $R_p$ for almost all $p$.
\end{enumerate}
\end{prop}

For every integer $n \ge 1$, we have a natural homomorphism 
\[
R[X_1,\dots,X_n]\to \ulim_p (R_p[X_1,\dots,X_n]).
\]

\begin{defn}
Let $(R_p)_{p\in\mathcal{P}}$ be an approximation of $R$ and $R_p[X_1, \dots, X_n]$ (resp.~$R[X_1, \dots, X_n]$) denote the polynomial ring in $n$ variables over $R_p$ (resp.~$R$). 
An \textit{approximation} of a polynomial $f \in R[X_1, \dots, X_n]$ is a family of polynomials $(f_p)_{p\in \mathcal{P}}$ such that $f_p \in R_p[X_1, \dots, X_n]$ for all $p \in \mathcal{P}$ and $f$ is mapped to $\ulim_p f_p$ under the natural ring homomorphism 
\[
R[X_1,\dots,X_n]\to \ulim_p (R_p[X_1,\dots,X_n]).
\]
An approximation of an ideal in $R[X_1, \dots, X_n]$ is defined in an analogous way.
\end{defn}

\begin{defn}
Let $S$ be a ring of finite type over $R$, that is, $S\cong R[X_1,\dots,X_n]/I$ for some ideal $I \subseteq R[X_1,\dots,X_n]$. 
\begin{enumerate}[label=(\roman*)]
\item 
A family of rings $(S_p)_{p\in \mathcal{P}}$ is said to be a \textit{relative approximation} of $S$ if 
\[
S_p=R_p[X_1,\dots,X_n]/I_p \text{ for almost all $p$},
\]
 where $(I_p)_{p\in \mathcal{P}}$ is an approximation of $I$. 
The ultraproduct $S_\infty:=\ulim_p S_p$ is called the \textit{relative hull} of $S$.
\item 
Let $M$ be a finitely generated $S$-module.
We can write $M$ as the cokernel of an $S$-linear map between finite free modules, i.e., there exists an exact sequence
\[
S^{m}\xrightarrow{A} S^{n}\to M \to 0
\]
such that $A$ is an $n\times m$-matrix with entries in $S$. 
Let $(S_p)_{p\in \mathcal{P}}$ be an approximation of $S$ and let $(A_p)_{p\in \mathcal{P}}$ be an entrywise approximation of $A$. 
A family $(M_p)_{p\in \mathcal{P}}$ of $S_p$-modules is said to be an \textit{approximation} of $M$ if $M_p$ is the cokernel of the $S_p$-linear map $S_p^m\xrightarrow{A_p} S_p^n$ for almost all $p$. 
\end{enumerate}
\end{defn}

\begin{rem}
The $S_\infty$-module $\ulim_p M_p$ is isomorphic to $M\otimes_S S_\infty$. This isomorphism  does not depend on the choice of presentation of $M$. 
\end{rem}

We next introduce ultra-Frobenii, which are an ultraproduct analog of Frobenius maps.
\begin{defn}
Let $S$ be a ring of finite type over $R$.
\begin{enumerate}[label=(\roman*)]
\item Let $\epsilon=\ulim_pe_p$ be a non-standard natural number. 
An \textit{ultra-Frobenius} 
\[
F^\epsilon:S\to S_\infty
\]
is the ring homomorphism that sends $x$ to $\ulim_p x_p^{p^{e_p}}$, where $(x_p)_{p\in\mathcal{P}}$ is an approximation of $x$. 
We also use the same notation $F^\epsilon$ to denote the ring homomorphism $S_\infty\to S_\infty$ that sends $\ulim_p x_p$ to $\ulim_p x_p^{p^{e_p}}$.
\item For an element $f=\ulim_p f_p\in S_\infty$ and a non-standard natural number $\epsilon$, we define $f^{\epsilon}:=\ulim_p f_p^{e_p}$. 
Note that if $\epsilon$ is a standard natural number, then $f^\epsilon$ agrees with the usual power of $f$.
\item For a non-standard natural number $\epsilon$, let $F^\epsilon_*$ denote the restriction of scalars along $F^\epsilon:S_\infty\to S_\infty$. 
For an $S_\infty$-module $M$ and an element $x\in M$, we write $F^\epsilon_*x$ to denote $x$ viewed as an element of $F^\epsilon_*M$.
\end{enumerate}
\end{defn}

Suppose that $(R,\m)$ is a $d$-dimensional local ring essentially of finite type over $\C$ and $x_1,\dots,x_d$ is a system of parameters for $R$. 
Let $(R_p)_{p\in\mathcal{P}}$ and $(x_p)_{p\in \mathcal{P}}$ be approximations of $R$ and $x:=x_1 \cdots x_d$, respectively, and let $(M_p)_{p\in\mathcal{P}}$ be a family of $R_p$-modules. 
The \v{C}ech complex of $M_p$ induces an exact sequence
\[
	\bigoplus_{i=1}^d(M_p)_{\check{x_{i,p}}} \to (M_p)_{x_{p}} \to H_{\m_p}^d(M_p)\to 0,
\]
where $(\check{x_{i,p}})_{p\in \mathcal{P}}$ is an approximation of $\check{x_{i}}:=x_1\cdots x_{i-1}x_{i+1} \cdots x_d$, for almost all $p$. 
Taking ultraproducts yields an exact sequence 
\[
\bigoplus_{i=1}^d\ulim_p (M_p)_{\check{x_{i,p}}} \to \ulim_p (M_p)_{x_{p}} \to \ulim_p H_{\m_p}^d(M_p)\to 0
\]
Then, from the commutative diagram
\[
	\xymatrix{
	\bigoplus_{i=1}^d (\ulim_p M_p)_{\check{x_i}} \ar[r] \ar[d] & (\ulim_p M_p)_{x} \ar[d] \\
	\bigoplus_{i=1}^d\ulim_p (M_p)_{\check{x_{i,p}}} \ar[r] & \ulim_p (M_p)_{x_{p}}
	},
\]
we have a canonical homomorphism
\[
	H_\m^d(\ulim_p M_p) \to \ulim_p H^d_{\m_p}(M_p).
\]
For an element 
\[
\eta=\left[\frac{y}{x^t}\right]\in H_\m^d(\ulim_p M_p) \quad (y \in \ulim_p M_p, t \in \N), 
\]
a family $(\eta_p)_{p\in\mathcal{P}}$ of elements of $H_{\m_p}^d(M_p)$ is called an \textit{approximation} of $\eta$ if 
\[
    \eta_p=\left[\frac{y_p}{x_p^t}\right] \in H^d_{\m_p}(M_p) 
\]
for almost all $p$. If $(\eta_p)_{p\in\mathcal{P}}$ is an approximation of $\eta$, then the image of $\eta$ under the above canonical homomorphism is equal to $\ulim_p \eta_p$.

\section{Local \texorpdfstring{$F$}{F}-alpha invariants}

In this section, we introduce a local version of Pande's $F$-alpha invariants (see \cite{Pande23})—a positive characteristic analog of Tian's alpha invariants—and study its basic properties.

\subsection{Definition and basic properties}


\begin{defn}\label{normalized order}
Let $(R,\m)$ be a Noetherian local domain and $x \in R$ be a nonzero element of $R$. 
The \textit{normalized order} $\overline{\mathrm{ord}}_{\m}(x)$ of $x$ with respect to the maximal ideal $\m$ is defined as 
\[
\overline{\mathrm{ord}}_{\m}(x):=\max\left\{\frac{a}{b} \ge 0 \;\middle|\; a\in \Z_{\ge0}, b\in \Z_{>0}, x^b \in \overline{\m^a}\right\}, 
\]
where $\overline{\m^a}$ is the integral closure of the ideal $\m^a$. 
Equivalently, 
\[\overline{\mathrm{ord}}_{\m}(x)
=\lim_{\ell\to\infty}\frac{\mathrm{ord}_{\m}(x^\ell)}{\ell}, 
\]
where $\mathrm{ord}_{\m}(x^{\ell}):=\max\{n \ge 0 \mid x^{\ell} \in \m^n\}$.
\end{defn}

\begin{rem}
We have the following. 
\begin{enumerate}
\item For every integer $n \ge 1$, we have $x \in \overline{\m^n}$ if and only if $\overline{\mathrm{ord}}_{\m}(x)\geq n$.
\item 
$\overline{\mathrm{ord}}_{\m}(x^{\ell})=\ell \;  \overline{\mathrm{ord}}_{\m}(x)$ for all integers $\ell \ge 1$. 
\item 
If $R$ is a regular local ring or the localization of a normal standard graded ring over a field at the unique homogeneous maximal ideal, then \cite[Proposition 2.10]{Polstra26} shows that 
$\overline{\mathrm{ord}}_{\m}(x)=\mathrm{ord}_{\m}(x)$. 
\end{enumerate}
\end{rem}

We now introduce the notion of local $F$-alpha invariants, which serve as a local analog of  Pande's $F$-alpha invariants \cite{Pande23}. 
\begin{defn} \label{local F-alpha invariant}
Let $(R,\m)$ be a Noetherian $F$-finite $F$-pure local domain of characteristic $p>0$ and of dimension $d \ge 1$. The \textit{local $F$-alpha invariant} $\alpha_F(R)$ of $R$ is defined as
\[
\alpha_F(R)=\inf_{f \in \m}\operatorname{fpt}(f)\overline{\mathrm{ord}}_{\m}(f),
\]
where $f$ runs through all nonzero elements of $\m$.
\end{defn}

\begin{rem}\label{alpha=0}
    If $R$ is not strongly $F$-regular, then $\alpha_F(R)=0$. 
    Indeed, by the definition of strong $F$-regularity, there exists a nonzero element $c\in \m$ such that for every $e\ge 1$, the map $R\xrightarrow{\cdot F^e_*c} F^e_*R$ fails to split. 
    Therefore, $\operatorname{fpt}(c)=0$, which shows that $\alpha_F(R)=0$.
\end{rem}

\begin{prop}\label{local alpha invariant of RLR}\label{local alpha invariant of completion}
Let $(R,\m)$ be a strongly $F$-regular local domain of characteristic $p > 0$ and of dimension $d \ge 1$.
\begin{enumerate}[label=$(\arabic*)$]
\item $0 \le \alpha_F(R) \le 1$. 
\item If $R$ is regular, then $\alpha_F(R)=1$.  
\item Let $\widehat{R}$ denote the $\m$-adic completion of $R$. Then 
$\alpha_F(R)=\alpha_F(\widehat{R})$.
\end{enumerate}
\end{prop}
\begin{proof}
(1) Choose a Rees valuation $v$ of $\m$ and an element $x\in \m$ such that $v(x) = v(\m)$. Then $\overline{\mathrm{ord}}_{\m}(x) = 1$, so
$\alpha_F (R) \le \mathrm{fpt}(x) \le 1$.

(2) If $R$ is regular, then for every nonzero element $f \in \m$, we have 
\[
\operatorname{fpt}(f) \overline{\mathrm{ord}}_{\m}(f)=\operatorname{fpt}(f) {\mathrm{ord}}_{\m}(f) \ge 1
\]
by \cite[Proposition 4.1]{TW04}.
Thus, $\alpha_F(R)=1$.

(3) For every nonzero element $f \in \m$, 
we have the equalities $\operatorname{fpt}(R,f)=\operatorname{fpt}(\widehat{R},f)$ (see Remark \ref{test ideal remark}) 
and $\overline{\mathrm{ord}}_{\m}(f)=\overline{\mathrm{ord}}_{\m \widehat{R}}(f)$ (see \cite[Proposition 1.6.2]{Huneke-Swanson}). Therefore, $\alpha_F(R)\ge \alpha_F(\widehat{R})$. 

Suppose that $\alpha_F(R)>\alpha_F(\widehat{R})$. Let $\epsilon>0$ be a real number such that 
\[
\epsilon<\alpha_F(R)-\alpha_F(\widehat{R}). 
\]
By the definition of $\alpha_F(\widehat{R})$, there exists a nonzero element $f\in \m \widehat{R}$ such that 
\[
\operatorname{fpt}(\widehat{R},f) \overline{\mathrm{ord}}_{\m \widehat{R}}(f)<\alpha_F(R)-\epsilon.
\]
Fix a sufficiently large integer $N$ such that 
\[\operatorname{fpt}(\widehat{R},\m \widehat{R})\overline{\mathrm{ord}}_{\m \widehat{R}}(f)<N\epsilon \quad \text{and} \quad \overline{\mathrm{ord}}_{\m \widehat{R}}(f)<N.\] 
Take an element $g\in R$ such that $f-g\in \m^{N} \widehat{R}$. Then by \cite[Proposition 4.4]{TW04}, 
	\begin{align*}
		\operatorname{fpt}(R,g) = \operatorname{fpt}(\widehat{R},g) 
		&\le \operatorname{fpt}(\widehat{R},f)+\operatorname{fpt}(\widehat{R},\m^N \widehat{R}) \\
		& < \operatorname{fpt}(\widehat{R},f)+\frac{\epsilon}{\overline{\mathrm{ord}}_{\m \widehat{R}}(f)}. 
	\end{align*}
Since $\overline{\mathrm{ord}}_{\m}(g)=\overline{\mathrm{ord}}_{\m \widehat{R}}(f)$ by the choice of $g$, it follows that  
\[
\operatorname{fpt}(R,g)\overline{\mathrm{ord}}_{\m}(g)<\operatorname{fpt}(\widehat{R}, f)\overline{\mathrm{ord}}_{\m \widehat{R}}(f)+\epsilon<\alpha_F(R),\]
which contradicts the definition of $\alpha_F(R)$. 
\end{proof}

Next, we present a lemma on $F$-pure thresholds that will be used to give an alternative description of local $F$-alpha invariants.

\begin{lem}[cf.~\cite{MTW05}]\label{fpt lem}
Let $(R,\m)$ be a strongly $F$-regular local domain of characteristic $p > 0$, and let $\ba$ be a nonzero ideal of $R$ contained in $\m$. 
\begin{enumerate}[label=$(\arabic*)$]
\item 
For each integer $e \ge 1$, set 
\[
\nu_e(\ba)=\max\{r \ge 0 \mid F^e_* \ba^r \cdot \Hom_R(F^e_*R, R) \to R \textup{ is surjective}\}, 
\]
where the map 
\[F^e_* \ba^r \cdot \Hom_R(F^e_*R, R) \to R\]
is the evaluation map sending $F^e_*x\cdot\varphi$ to $\varphi(F^e_*x)$. 
Here we view $\Hom_R(F^e_*R,R)$ as an $F^e_*R$-module, so that $F^e_*\ba^r \cdot \Hom_R(F^e_*R,R)$ is an $F^e_*R$-submodule of $\Hom_R(F^e_*R,R)$. 
Then the sequence $(\nu_e(\ba)/p^e)_{e \ge 1}$ is monotonically increasing and 
\[
\operatorname{fpt}(\ba)=\sup_e \frac{\nu_e(\ba)}{p^e}. 
\]
Moreover, for all integers $e \ge 1$, 
\[
\operatorname{fpt}(\ba)>\frac{\nu_e(\ba)}{p^e}.  
\]
\item 
Suppose that $\ba=(f)$ is a principal ideal. For each integer $e \ge 1$, set 
\begin{align*}
\mu_e(f)&=\nu_e(f)+1\\
&=\min\{r \ge 0 \; | \; F^e_* f^r \cdot \Hom_R(F^e_*R, R) \to R \textup{ is not surjective}\}.
\end{align*}
Then the sequence $(\mu_e(f)/p^e)_{e \ge 1}$ is monotonically decreasing and 
\[
\operatorname{fpt}(f)=\inf_e \frac{\mu_e(f)}{p^e}.
\]
\end{enumerate}
\end{lem}
\begin{proof}
The proof closely follows that of \cite{MTW05}, although the ring is assumed to be regular in \cite{MTW05}.   
First, recall from \cite[Proposition 4.4]{TakTak08} that 
\[
\operatorname{fpt}(\ba)=\lim_{e \to \infty}\frac{\nu_e(\ba)}{p^e} \text{ and } \operatorname{fpt}(f)=\lim_{e \to \infty}\frac{\mu_e(f)}{p^e}.
\]

For the first part of (1), it suffices to show that $p \nu_e(\ba) \le \nu_{e+1}(\ba)$ for every integer $e \ge 1$. 
For (2), it suffices to show that $p \mu_e(f) \ge \mu_{e+1}(f)$ for every integer $e \ge 1$. 
These inequalities are an immediate consequence of the following elementary fact: 
for any ideal $\mathfrak{b}$ in $R$, 
the map $F^{e+1}_*\mathfrak{b}^{[p]} \cdot \Hom_R(F^{e+1}_*R, R) \to R$ factors through $F^{e}_*\mathfrak{b} \cdot \Hom_R(F^{e}_*R, R) \to R$, and since $R$ is $F$-pure, the map 
\[F^{e+1}_*\mathfrak{b}^{[p]} \cdot \Hom_R(F^{e+1}_*R, R) \to F^e_*\mathfrak{b} \cdot \Hom_R(F^{e}_*R, R)\] 
is surjective. 

For the second part of (1), it suffices to show that there exists an integer $e_0 \ge 1$ such that $p^{e_0} \nu_e(\ba) < \nu_{e+e_0}(\ba)$ for every integer $e \ge 1$. 
Since $R$ is strongly $F$-regular, there exists an integer $e_0 \ge 1$ such that $F^{e_0}_*\mathfrak{a} \cdot \Hom_R(F^{e_0}_*R, R) \to R$ is surjective. Then for all integers $e \ge 1$ and $r \ge 0$, the map 
\[F^{e+e_0}_*\mathfrak{a}^{p^{e_0}r+1} \cdot \Hom_R(F^{e+e_0}_*R, R) \to F^e_*\mathfrak{a}^r \cdot \Hom_R(F^{e}_*R, R)\] 
is surjective. 
This implies the desired inequality. 
\end{proof}

\begin{prop}\label{alternative alpha_F}
Let $(R,\m)$ be a strongly $F$-regular local domain of characteristic $p > 0$ and of dimension $d \ge 1$.
For every integer $e \ge 1$, consider the ideal $I_e(R) \subseteq R$ defined by  
\[
I_e=I_e(R):=\{c \in R \mid F^e_*c \cdot \Hom_R(F^e_*R, R) \to R \textup{ is not surjective}\}. 
\]
Let $m_e(R)$ be the largest integer $m$ such that $I_e(R) \subseteq \overline{\m^m}$. Then 
\[
\alpha_F(R) = \lim_{e \to \infty}\frac{m_e(R)}{p^e}.
\]
\end{prop}

\begin{proof}
First, we show that $\alpha_F(R) \le \liminf_{e \to \infty}{m_e(R)}/{p^e}$. 
For every integer $e \ge 1$, by the definition of $m_e(R)$, there exists an element $c_e \in I_e$ not contained in $\overline{\m^{m_e(R)+1}}$. 
Then $\mu_e(c_e)=1$ and $\lfloor \overline{\mathrm{ord}}_{\m} (c_e)\rfloor =m_e(R)$. 
It follows from Lemma \ref{fpt lem} (2) that 
\[
\operatorname{fpt}(c_e)\overline{\mathrm{ord}}_{\m}(c_e) \le  \frac{m_e(R)+1}{p^e}, 
\]
which implies the desired inequality. 

Next, we show that $\alpha_F(R) \ge \limsup_{e \to \infty}{m_e(R)}/{p^e}$.
Fix a nonzero element $f \in \m$. 
Since $f^{\mu_e(f)} \in I_e \subseteq \overline{\m^{m_e(R)}}$ by definition, we have 
\[
\mu_e(f) \overline{\mathrm{ord}}_{\m}(f)=\overline{\mathrm{ord}}_{\m}(f^{\mu_e(f)}) \ge m_e(R).
\]
Dividing both sides by $p^e$ and taking the limit superior, we obtain  
\[
\operatorname{fpt}(f)\overline{\mathrm{ord}}_{\m}(f) \ge \limsup_{e \to \infty}\frac{m_e(R)}{p^e}.
\]
\end{proof}

\begin{rem}\label{F-signature remark}
$F$-signature can be described in terms of the ideals $I_e$. 
It follows from \cite[Proposition 4.5]{Tucker12} that $s(R)=\lim_{e \to \infty}\ell_R(R/I_e)/p^{de}$. 
We use this description in the following proposition. 
\end{rem}

We follow Pande's argument to establish basic bounds for local $F$-alpha invariants. 

\begin{prop}[\textup{cf.~\cite[Theorem 4.6]{Pande23}}]\label{comparison of F-signature and local alpha invariant}
Let $(R,\m)$ be a strongly $F$-regular local domain of characteristic $p > 0$ and of dimension $d \ge 1$.
Let $e(R)$ denote the Hilbert--Samuel multiplicity of $R$. 
\begin{enumerate}[label=$(\arabic*)$]
\item The following inequalities hold:
\[
 \frac{e(R) \alpha_F(R)^d}{d!} \le s(R) \le \frac{e(R)\operatorname{fpt}(\m)^d}{d!}.
\]
\item
Let $C:=\operatorname{edim}(R)+1$ and let $w$ be a Rees valuation centered at the maximal ideal $\m$. 
By Izumi's theorem $($see \cite[Theorem~1.2]{Hubl-Swanson}$)$, there exists an integer $N \ge 1$ such that 
\[
w(f)\le N\overline{\operatorname{ord}}_{\m}(f)\le Nw(f)
\]
for every element $f\in R$. Then  
	\[
		s(R) \le \frac{e(R)N^d}{d!}(C^d-(C-\alpha_{F}(R))^d). 
	\]
\end{enumerate}
\end{prop}
\begin{proof}
(1) First, we show the upper bound for $s(R)$. 
By Lemma \ref{fpt lem} (1), we have 
\[
\operatorname{fpt}(\m)p^e>\nu_e(\m),
\]
so that $\m^{\lceil \operatorname{fpt}(\m)p^e \rceil} \subseteq I_e$. Then by Remark \ref{F-signature remark}, 
\begin{align*}
s(R)=\lim_{e \to \infty} \frac{\ell_R(R/I_e)}{p^{de}} &\le \lim_{e \to \infty} \frac{\ell_R(R/\m^{\lceil \operatorname{fpt}(\m)p^e \rceil})}{p^{de}}\\
&=\lim_{e \to \infty} \frac{1}{d!}\frac{\ell_R(R/\m^{\lceil \operatorname{fpt}(\m)p^e \rceil})d!}{{\lceil \operatorname{fpt}(\m)p^e \rceil}^d}\left(\frac{{\lceil \operatorname{fpt}(\m)p^e \rceil}}{p^{e}}\right)^d\\
&=\frac{1}{d!}e(R)\operatorname{fpt}(\m)^d. 
\end{align*}

For the lower bound for $s(R)$, if $\alpha_F(R)=0$, then the inequality is clear. 
Hence, we may assume that $\alpha_F(R)>0$, and then, by Proposition \ref{alternative alpha_F}, we have $m_e(R)\to \infty$ as $e\to \infty$. 
Note that the zeroth normal Hilbert coefficient 
\[
\overline{e}_0(\m)=\lim_{n \to \infty} d! \; \frac{\ell_R(R/\overline{\m^n})}{n^d}
\] 
of $\m$ is nothing but the Hilbert--Samuel multiplicity $e(R)$ of $R$.
Then, using Remark \ref{F-signature remark} again, we have 
\begin{align*}
s(R)=\lim_{e \to \infty} \frac{\ell_R(R/I_e)}{p^{de}} &\ge \lim_{e \to \infty} \frac{\ell_R(R/\overline{\m^{m_e(R)}})}{p^{de}}\\
&=\lim_{e \to \infty} \frac{1}{d!}\frac{\ell_R(R/\overline{\m^{m_e(R)}})d!}{m_e(R)^d}\left(\frac{m_e(R)}{p^{e}}\right)^d\\
&=\frac{1}{d!}e(R)\alpha_F(R)^d. 
\end{align*}

(2)Let $\pi:Y\to \Spec R$ be the normalized blow-up of $\m$ and write
\[
    \m \sO_Y=\sO_Y(-\sum_{i=1}^r a_iE_i),
\]
where the $E_i$ are distinct prime Weil divisors on $Y$ and the $a_i$ are positive integers. We may assume that $w=\mathrm{ord}_{E_1}/a_1$. 
By Izumi's theorem \cite[Theorem 1.2]{Hubl-Swanson} and the description of integral closure in terms of Rees valuations (see the paragraph preceding \cite[Theorem 2.6]{Hubl-Swanson}; cf.~\cite[Theorem 4.3 (i)]{BFJ14}), there exists a positive integer $N$ such that for every $f\in R$, we have 
\[w(f)\le N\overline{\mathrm{ord}}_{\m}(f)\le N w(f).\] 
For each rational number $t \ge 0$, let $\ba_t:=\{f\in R\mid w(f)\ge t\}$, and define the invariant $\alpha_{F,w}(R)$ associated with $w$ by 
\[
    \alpha_{F,w}(R):=\inf_{f\in \m} \operatorname{fpt}(f)w(f),
\]
where $f$ runs through all nonzero elements of $\m$.
It is easy to see that $\ba_t$ is an ideal of $R$ and that $0 \le \alpha_{F,w}(R) \le \min\{1, N\alpha_F(R)\}$. 

Let $0<\epsilon<1$ be arbitrary.
By the definition of $\alpha_{F,w}(R)$, there exists a nonzero element $f\in \m$ such that $\operatorname{fpt}(f)w(f)<\alpha_{F,w}(R)+\epsilon$. 
Then we can choose positive integers $a$ and $e_0$ such that 
\[
    \operatorname{fpt}(f)<\frac{a}{p^{e_0}-1}<\frac{\alpha_{F,w}(R)+\epsilon}{w(f)}.
\]
Replacing $f$ by $f^a$, we may assume that $a=1$.
For each integer $s\ge w(f)$, set 
\[
v_s:=\frac{p^{se_0}-1}{p^{e_0}-1}.
\]
Since $1/(p^{e_0}-1)>\operatorname{fpt}(f)$, it follows from \cite[Lemma 3.11]{Pande23} that $f^{v_s}\in I_{se_0}(R)$. 
Moreover, for all sufficiently large $e$, the Brian\c{c}on--Skoda theorem gives
\[
\overline{\m^{Cp^e}}\subseteq \m^{\operatorname{edim}(R)p^e}\subseteq \m^{[p^e]}\subseteq I_e(R).
\]
For $s \gg 0$, we have 
\[
NCp^{s e_0} \ge 2p^{s e_0} >(\alpha_{F,w}(R)+\epsilon)p^{s e_0}>\frac{w(f)}{p^{e_0}-1}p^{s e_0}>w(f)v_s. 
\]
Setting $M(s):=NCp^{s e_0}-w(f)v_s >0$, we consider the exact sequence
\[
0 \to R/\ba_{M(s)} \xrightarrow{\cdot f^{v_s}} R/\ba_{NCp^{se_0}} \to R/(f^{v_s}R+\ba_{NCp^{se_0}}) \to 0.
\] 
By the choice of $N$ and the containments above, we have 
\[f^{v_s}R+\ba_{NCp^{se_0}}\subseteq f^{v_s}R+\overline{\m^{Cp^{se_0}}}\subseteq I_{se_0}(R).\] 
Together with the exact sequence above, this gives 
\begin{align*}
\ell_R(R/I_{se_0}(R)) & \le \ell_R(R/(f^{v_s}R+\ba_{NCp^{se_0}}))\\
& =\ell_R(R/\ba_{NCp^{se_0}})-\ell_R(R/\ba_{M(s)}).
\end{align*}
Therefore, 
	\begin{align*}
		s(R) &= \lim_{s\to \infty} \frac{1}{p^{se_0d}}\left(\ell_R(R/I_{se_0}(R))\right)\\
		&\le \lim_{s\to \infty} \frac{1}{p^{se_0d}}\left(\ell_R(R/\ba_{NCp^{se_0}})-\ell_R(R/\ba_{M(s)})\right)\\
		&=\frac{\operatorname{vol}(w)}{d!}\left(N^dC^d-\left(NC-\frac{w(f)}{p^{e_0}-1}\right)^d\right)\\
		&\le \frac{\operatorname{vol}(w)}{d!}(N^dC^d-(NC-(\alpha_{F,w}(R)+\epsilon))^d) \\
        &\le \frac{e(R)}{d!}(N^dC^d-(NC-(\alpha_{F,w}(R)+\epsilon))^d).
	\end{align*}
Here $\operatorname{vol}(w)$ denotes the volume of the valuation $w$, defined by 
\[
\operatorname{vol}(w):=\lim_{t\to \infty}\frac{d!}{t^d}\ell_R(R/\ba_t),
\]
as in \cite{ELS03}; see also \cite{Cutkosky13}. Letting $\epsilon$ tend to zero yields
\begin{align*}
    s(R) &\le \frac{e(R)}{d!}(N^dC^d-(NC-\alpha_{F,w}(R))^d) \\
    &\le \frac{e(R)N^d}{d!}(C^d-(C-\alpha_{F}(R))^d),
\end{align*}
where the last inequality follows from the inequality $\alpha_{F,w}(R)\le N\alpha_F(R)$.
\end{proof}

\begin{rem}
Suppose that $(R,\m)$ is a local ring essentially of finite type over $\C$ whose spectrum is of klt type.  
By the proof of Lemma \ref{normalized order under pure mophisms}, the integer $N$ in Proposition \ref{comparison of F-signature and local alpha invariant}~(2) for a reduction $R_{\mu}$ of $R$ modulo $p$ can be chosen uniformly for general closed points $\mu$.
\end{rem}

\begin{rem}
Let $(R,\m)$ be a Noetherian $F$-finite $F$-pure local domain of characteristic $p>0$ and of dimension $d \ge 1$. 
Remark \ref{alpha=0} and Proposition \ref{comparison of F-signature and local alpha invariant} imply that $R$ is strongly $F$-regular if and only if $\alpha_F(R)>0$.  
\end{rem}




\subsection{Local versus graded settings}
In this subsection, we compare our local invariants with the graded invariants introduced by Pande.
Throughout this subsection, we assume that $R=\bigoplus_{i \ge 0} R_i$ is a strongly $F$-regular graded domain of dimension $d \ge 1$ and of finite type over an $F$-finite field $R_0=k$ of characteristic $p>0$, and let $\m=\bigoplus_{i \ge 1} R_i$ denote the irrelevant ideal of $R$. 

Pande \cite{Pande23} defined ``$\alpha_F$-invariants" for section rings, but her definition applies more generally to graded rings over an $F$-finite field. 
\begin{defn}[cf.~\textup{\cite[Definition 3.4]{Pande23}}]
Let $\Delta$ be an effective homogeneous $\Q$-Weil divisor on $\Spec R$. Then the \textit{$F$-alpha invariant} $\alpha_F(R,\Delta)$ of the pair $(R,\Delta)$ is defined as 
    \[
        \alpha_F(R,\Delta)=\inf_{f} \operatorname{fpt}((R,\Delta);f)\deg(f),
    \]
where $f$ runs through all nonzero homogeneous elements of $\m$. 
If $\Delta=0$, then this invariant is said to be the \textit{$F$-alpha invariant} of $R$ and simply denoted by $\alpha_F(R)$.
\end{defn}


To study the properties of $F$-alpha invariants, we consider the graded version of the ideals $I_e(R)$ introduced in Proposition \ref{alternative alpha_F}.

\begin{defn}
For every integer $e\ge 1$, the ideal $I_e(R) \subseteq R$ is defined as
	\[
		I_e(R)=\{x\in R \mid \text{$\phi(x)\in \m$ for all $\phi\in \Hom_R(F^e_*R,R)$}\}.
	\]
\end{defn}

\begin{rem}\label{graded I_e remark}
\begin{enumerate}
\item (cf.~\cite[Proposition 3.7]{SB18}) $I_e(R)$ is a homogeneous ideal.
\item (cf.~\cite[Proposition 3.10]{SB18}) $F$-pure thresholds of $R$ can be described in terms of $I_e(R)$ just as in the local case (see Lemma \ref{fpt lem}). 
Let $\ba$ be a nonzero homogeneous ideal of $R$. 
For each integer $e \ge 1$, define 
\[\nu_e(\ba)=\max\{r \ge 0 \; | \; \ba^r \not\subseteq I_e(R)\}.\] 
Then 
\[
\operatorname{fpt}(\ba)=\lim_{e \to \infty}\frac{\nu_e(\ba)}{p^e}.
\]
\item $I_e(R)=I_e(R_\m)\cap R$, where the definition of $I_e(R_{\m})$ can be found in Proposition \ref{alternative alpha_F}. 
\end{enumerate}
\end{rem}
\begin{lem}\label{fpt of non homogeneous elements}
Let $f\in \m$ be a nonzero element and let $f_0$ denote the lowest nonzero homogeneous component of $f$. Then $\operatorname{fpt}(R_\m, f)\ge \operatorname{fpt}(R_\m, f_0)$.
\end{lem}
\begin{proof}
For every integer $e \ge 1$ and every nonzero element $g \in \m$, set 
\[
\nu_e(R_{\m}, g):=\max\{r \ge 0 \; | \; g^r \notin I_e(R_{\m})\}.
\]
Note that $f_0^{\nu_e(R_{\m}, f_0)} \notin I_e(R_{\m}) \cap R=I_e(R)$. 
As $I_e(R)$ is a homogeneous ideal and $f_0^{\nu_e(R_{\m}, f_0)}$ is a homogeneous component of $f^{\nu_e(R_{\m}, f_0)}$, we see that $f^{\nu_e(R_{\m}, f_0)} \notin I_e(R)$, that is, $\nu_e(R_{\m}, f) \ge \nu_e(R_{\m}, f_0)$. 
It then follows from Lemma \ref{fpt lem} that 
\[
\operatorname{fpt}(R_\m, f_0)=\lim_{e \to \infty}\frac{\nu_e(R_{\m}, f_0)}{p^e} \le \lim_{e \to \infty}\frac{\nu_e(R_{\m}, f)}{p^e}=\operatorname{fpt}(R_\m, f). 
\]
\end{proof}

\begin{prop} \label{equivalence Pande's F-alpha invariant}
Suppose that $R$ is standard graded. Then
    \[
        \alpha_F(R)= \alpha_F(R_\m). 
    \]
\end{prop}
\begin{proof}
It follows that 
	\begin{align*}
		\alpha_F(R_\m) &= \inf_{f \in \m} \operatorname{fpt}(R_\m, f) \overline{\mathrm{ord}}_{\m R_\m}(f) \\
		&= \inf_{g \in \m \text{:homog}} \operatorname{fpt}(R_\m, g) \overline{\mathrm{ord}}_{\m R_\m}(g) \\
		&=  \inf_{g \in \m\text{:homog}} \operatorname{fpt}(R_\m, g) \deg(g) \\
            &= \alpha_F(R), 
	\end{align*}
where $g$ runs through all nonzero homogeneous elements of $\m$. 
Here, the second and third equalities use Lemma \ref{fpt of non homogeneous elements} and the formula $\overline{\mathrm{ord}}_{\m R_\m}(g)=\mathrm{ord}_{\m R_\m}(g)=\deg(g)$ given in \cite[Proposition 2.10]{Polstra26}.
The fourth equality follows from the fact that $\operatorname{fpt}(R_{\m},g)=\operatorname{fpt}(R, g)$ for every nonzero homogeneous element $g \in \m$, which is a consequence of Lemma \ref{fpt lem} and Remark \ref{graded I_e remark}. 
\end{proof}

\begin{prop}[{cf.~\cite[Proposition 4.12]{Pande23}}] \label{F-alpha invariant Veronese}
For every integer $n \ge 1$, let 
\[
R^{(n)}=\bigoplus_{i \ge 0}R_{in}
\]
denote the $n$-th Veronese subring of $R$. Then $\alpha_F(R)\le n\alpha_F(R^{(n)}).$
    Moreover, if $n$ is sufficiently divisible, then $n\alpha_F(R^{(n)})$ stabilizes.
\end{prop}
\begin{proof}
Let $\m^{(n)}=\bigoplus_{i \ge 1}R_{in}$ denote the irrelevant ideal of $R^{(n)}$. 
To prove the first assertion, it suffices to show that $\operatorname{fpt}(R,f)\le \operatorname{fpt}(R^{(n)},f)$ for every nonzero homogeneous element $f \in \m^{(n)}$. 
By Remark \ref{graded I_e remark}, it is enough to verify that for any positive integers $a, e$, if $f^a\in I_e(R^{(n)})$, then $f^a\in I_{e}(R)$. 
Let $\iota:R^{(n)}\to R$ be the natural inclusion and $\pi:R\to R^{(n)}$ the projection. 
Consider the map 
\[\Phi:\Hom_R(F^e_*R,R) \to \Hom_{R^{(n)}}(F^e_*R^{(n)},R^{(n)}); \quad \phi\mapsto \pi\circ \phi\circ F^e_*\iota.\] 
If $f^a\in I_e(R^{(n)})$, then $\Phi(\phi)(F^e_*f^a)\in \m^{(n)}$ for all  $\phi\in\Hom_R(F^e_*R,R)$. 
This implies that $\phi(F^e_*f^a)\in \m$, and therefore $f^a\in I_e(R)$.

Next, we prove the latter assertion. 
By the Pinkham–Demazure construction (see, for example, \cite{Watanabe81}), there exists an ample $\Q$-Cartier $\Q$-Weil divisor $D$ on $X:=\Proj R$ such that $R$ is isomorphic to the graded ring $\bigoplus_{i\ge 0}H^0(X,\sO_X(\lfloor iD\rfloor))$. 
Choose an integer $N \ge 1$ such that $ND$ is an ample Cartier divisor. 
Then, for every integer $l \ge 1$, it follows from \cite[Proposition 4.10]{Pande23} that  $l\alpha_F(R^{(lN)})=\alpha_F(R^{(N)})$, which completes the proof.
\end{proof}

In Proposition \ref{F-alpha invariant Veronese}, the equality does not hold in general.
\begin{eg}
Let $k$ be a perfect field of characteristic $p>5$ and let 
\[R=k[X,Y,Z]/(X^2+Y^3+Z^5)\] 
be a graded $k$-algebra with $\deg x=15$, $\deg y=10$, $\deg z=6$, where $x, y, z$ denote the images of $X,Y, Z$, respectively. 
Since $\operatorname{fpt}(R,z)\le 1/6$, we have 
\[\alpha_F(R)\le \operatorname{fpt}(R,z)\deg z\le\frac{1}{6} \cdot 6=1.\] 
On the other hand, the second Veronese subring $R^{(2)}$ is the polynomial ring $k[y,z]$, so we have $\operatorname{fpt}(R^{(2)},z)=1$. 
Therefore, $\alpha_F(R^{(2)})\le \operatorname{fpt}(R^{(2)},z)\deg(z)=1\cdot 3=3$. 

We now show the converse inequality. 
Let $f$ be a nonzero homogeneous element of $\m^{(2)}$ and let $t \in (0, 3/\deg f)$ be any real number.  
For all sufficiently large $e\gg 0$, we see that 
\[
\deg f\lfloor t(p^e-1) \rfloor
\le 3(p^e-1), 
\]
which implies that $f^{\lfloor t(p^e-1) \rfloor}\notin (y,z)^{[p^e]}R^{(2)} = I_e(R^{(2)})$. 
It then follows from the Fedder type criterion (see, for example, \cite[Lemma 3.9]{Takagi_invent}) that $\operatorname{fpt}(R^{(2)},f) \ge t$. 
Since $t$ is arbitrary in the interval $(0, 3/\deg f)$, 
we conclude that $\operatorname{fpt}(R^{(2)},f) \deg f \ge 3$, and hence $\alpha_F(R^{(2)})=3$. 
Thus, $\alpha_F(R)\le 1<6=2\alpha_F(R^{(2)})$.
\end{eg}

 We also introduce $F$-alpha invariants for projective varieties.
\begin{defn}[cf.~\textup{\cite[Definition 5.2]{Pande23}}]\label{projective alpha}
Suppose that $(X, \Delta)$ is a globally $F$-regular pair, where $X$ is a normal projective variety over an $F$-finite field of characteristic $p>0$ and $\Delta$ is an effective $\Q$-Weil divisor on $X$. 
Given a big $\Q$-Weil divisor $\Gamma$ on $X$, we define the \textit{$F$-alpha invariant} $\alpha_F((X,\Delta);\Gamma)$ of $(X, \Delta)$ with respect to $\Gamma$ as
\[
\alpha_F((X,\Delta);\Gamma)=\inf_{D}\mathrm{gfst}((X,\Delta);D),
\]
where $D$ ranges over all effective $\Q$-Weil divisors $\Q$-linearly equivalent to $\Gamma$. 
When $\Delta=0$, we write $\alpha_F(X;\Gamma)$ instead. 
Moreover, $\alpha_F(X;-K_X)$ is simply denoted by $\alpha_F(X)$.
\end{defn}

\begin{rem}
We can define a positive characteristic analog of $\delta$-invariants in a way similar to Definition \ref{projective alpha}. 
However, this definition differs from that given by Hattori--Odaka \cite[Definition 2.16]{Odaka-Hattori}. 
Their definition involves $F$-pure thresholds, whereas ours uses global $F$-split thresholds, which are in general strictly smaller than $F$-pure thresholds (see Remark \ref{fpt vs gfst}). 
\end{rem}

We relate the $F$-alpha invariant for a projective variety to that for its section ring. 
\begin{prop}[\textup{cf.~\cite[Lemma 3.7]{Pande23}}]\label{characterization of alpha invariant using global fpt}
With notation as in Definition \ref{projective alpha}, assume that $\Gamma$ is an ample $\Q$-Cartier $\Q$-Weil divisor on $X$ and set $S:=\bigoplus_{i\ge 0} H^0(X,\sO_X(\lfloor i\Gamma \rfloor))$. 
Fix an integer $r \ge 1$ such that $r\Gamma$ is Cartier, and consider the $r$-th Veronese subring $S^{(r)}$ of $S$.  
Let $\Delta_{S^{(r)}}$ denote the effective $\Q$-Weil divisor on $\Spec S^{(r)}$ corresponding to $\Delta$. Then 
\[
\alpha_F((X,\Delta);\Gamma)=r\alpha_F(S^{(r)},\Delta_{S^{(r)}}).
\]
\end{prop}
\begin{proof}
For a real number $t>0$, the condition $t \le \alpha_F((X,\Delta);\Gamma)$ holds if and only if the pair $(X,\Delta+(t/(nr)-\epsilon)\Div f)$ is globally $F$-regular for all integers $n \ge 1$, all sufficiently small real numbers $\epsilon>0$ and all nonzero elements $f \in H^0(X,\sO_X(rn\Gamma))$. 
By \cite[Proposition 5.3]{SS10}, the latter condition is equivalent to saying that the pair $(S^{(r)},\Delta_{S^{(r)}}+(t/(nr)-\epsilon)\Div f)$ is strongly $F$-regular for all such $n, \epsilon$ and $f$, which in turn holds if and only if $t \le r\alpha_F(S^{(r)},\Delta_{S^{(r)}})$.
\end{proof}
\section{\texorpdfstring{$F$}{F}-signature in equal characteristic zero}\label{F-signature char 0 section}

In this section, we define the notion of $F$-signature in equal characteristic zero, and study its positivity using ultraproducts. 

\begin{defn}\label{F-signature in char0}
Let $(R,\m)$ be a local domain essentially of finite type over $\C$.  
Suppose that $(A, R_A)$ is a model of $R$. 
Then {\it $F$-signature} $s(R)$ of $R$ is defined as
	\[
		s(R)=\sup_{U}\inf_{\mu\in U} s(R_{\mu}), 
	\]
	where $U$ runs through all dense open subsets of $\Spec A$ and $\mu$ runs through all closed points of $U$. 
	This definition is independent of the choice of the model.
\end{defn}

\begin{prop}
Definition \ref{F-signature in char0} is independent of the choice of model. 
\end{prop}

\begin{proof}
For a model $(A, R_A, \m_A)$ of $(R,\m)$, we temporarily write 
\[s_A(R):=\sup_U \inf_{\mu \in U} s(R_\mu).\] 
Let $(B, R_B,\m_B)$ be another model of $(R,\m)$. 
We will show that $s_A(R)=s_B(R)$. 
The proof is similar to an argument in \cite[Remark 2.5]{MS11}. 

First note that $s_A(R)$ remains unchanged after localizing $A$ at a single element. 
Let $x \in X$ be a point of a complex algebraic variety $X$ such that $\sO_{X,x} \cong R$.
Suppose that $(x_A \in X_A)$ is a model of $(x \in X)$ over $A$. 
Thus, localizing $A$ at a single element, we may assume that $\m_A(R_A \otimes_A A/\mu)$ is a prime ideal for all closed points $\mu \in \Spec A$. 

We can take a larger model $(C, R_C, \m_C)$ satisfying the following conditions: 
\begin{enumerate}[label=(\roman*)]
    \item $A \subseteq C$, 
    \item $R_C \cong R_A \otimes_A C$, 
    \item $\m_A R_C=\m_C$,  
    \item $\m_C(R_C \otimes_C C/\nu)$ is a prime ideal for all closed points $\nu \in \Spec C$. 
\end{enumerate}
It then suffices to show that $s_A(R)=s_C(R)$. 
Let $\varphi : \Spec C \to \Spec A$ be the morphism induced by the inclusion $A \hookrightarrow C$. 
Since $C$ is of finite type over $\Z$, every closed point $\nu \in \Spec C$ is sent by $\varphi$ to a closed point $\mu:=\varphi(\nu) \in \Spec A$. 
Therefore, we have a faithfully flat ring homomorphism 
\[R_A \otimes_A A/\mu \to R_A \otimes_A C/\nu \cong R_C \otimes_C C/\nu,\] 
which induces a flat local homomorphism 
\[
R_{\mu}:=\left(R_A \otimes_A A/\mu\right)_{\m_A(R_A \otimes_A A/\mu)} \to R_{\nu}:=\left(R_C \otimes_C C/\nu\right)_{\m_C(R_C \otimes_C C/\nu)}. 
\]
The closed fiber of the map $R_{\mu} \to R_{\nu}$ is a field, so it follows from \cite[Theorem 5.6]{Yao06} that $s(R_{\mu})=s(R_{\nu})$. 

We now prove that $s_A(R)=s_C(R)$.
The inequality $s_A(R) \le s_C(R)$ can be checked easily.
Indeed, for every dense open subset $U$ of $\Spec A$, set $V:=\varphi^{-1}(U)$. 
Since every closed point in $V$ is sent to a closed point in $U$, $\inf_{\mu \in U}s(R_{\mu}) \le \inf_{\nu \in V}s(R_{\nu})$,
where $\mu$ and $\nu$ run through all closed points of $U$ and $V$, respectively. 
Therefore, $s_A(R) \le s_C(R)$. 

We show the converse inequality $s_A(R) \ge s_C(R)$. 
Let $V$ be a dense open subset of $\Spec C$, and take a dense open subset $U$ of $\Spec A$ such that $U \subseteq \varphi(V)$; such a $U$ exists because $\varphi$ is dominant. 
For every closed point $\mu$, the fiber $V \cap \varphi^{-1}(\mu)$ contains a closed point $\nu$ of $\Spec C$. Hence, $\inf_{\mu \in U}s(R_{\mu}) \ge \inf_{\nu \in V}s(R_{\nu})$,
where $\mu$ and $\nu$ run through all closed points of $U$ and $V$, respectively. 
Thus, $s_A(R) \ge s_C(R)$. 
\end{proof}

\begin{rem}
Suppose that $R=(\C[x_1, \dots, x_n]/(f_1, \dots, f_r))_{(x_1, \dots, x_n)}$ is the localization at the maximal ideal $(x_1, \dots, x_n)$ of a quotient ring of a polynomial ring $\C[x_1, \dots, x_n]$ over $\C$ and that all coefficients of the defining equations $f_i$ are integers. 
For each prime number $p$, let $R_p=(\F_p[x_1, \dots, x_n]/(\overline{f_1}, \dots, \overline{f_r}))_{(x_1, \dots, x_n)}$, where $\overline{f_i}$ is the image of $f_i$ in the polynomial ring $\F_p[x_1, \dots, x_n]$ over $\F_p$. Then 
	\[
		s(R)=\liminf_{p\to \infty} s(R_p).  
	\]
\end{rem}

We record some basic properties of $F$-signature in equal characteristic zero. 

\begin{prop} \label{prop F-signature in characteristic 0}
Let $(R,\m)$ be a $d$-dimensional local domain essentially of finite type over $\C$.
\begin{enumerate}[label=$(\arabic*)$]
\item If $s(R)>0$, then $R$ is of strongly $F$-regular type. 
\item $s(R)=1$ if and only if $R$ is regular. 
\item Suppose that $R$ is normal $\Q$-Gorenstein and $f\in \m$ such that $R/(f)$ is a domain. 
Then
		\[
			s(R)\ge s(R/(f)).
		\]
	\end{enumerate}
\end{prop}
\begin{proof}
We use the same notation as in Definition \ref{F-signature in char0}. 
\begin{enumerate}
\item 
This is immediate from Proposition \ref{F-signature basic} (3). 
\item 
If $R$ is regular, then there exists a dense open subset $U$ of $\Spec A$ such that $R_{\mu}$ is regular for every closed point $\mu \in U$. 
Therefore, by Proposition \ref{F-signature basic} (2), $s(R)=1$. 

For the converse, suppose that $s(R)=1$ and $R$ is not regular. 
Then by (1), there exists a dense open subset $U \subseteq \Spec A$ such that $R_{\mu}$ is a $d$-dimensional, $F$-finite, normal, Cohen--Macaulay, non-regular local domain for every closed point $\mu \in U$. 
It follows from \cite[Theorem 4.14]{Tucker12} (see also \cite[Proposition 14]{HL02}) that  
\[
s(R_{\mu})\le \frac{e(R_{\mu})-e_{\mathrm{HK}}(R_{\mu})}{e(R_{\mu})-1}.
\]
On the other hand, by \cite[Theorem 4.12]{AE08}, 
\[
e_{\mathrm{HK}}(R_{\mu})\ge 1+\frac{1}{d\cdot (d!(d-1)+1)^d}.
\]
This contradicts the assumption that $s(R)=1$. Thus, $R$ is regular.

\item 
Let $S:=R/(f)$. 
By (1), we may assume that $S$ is of strongly $F$-regular type. 
After enlarging $A$ if necessary, we may choose a model $S_A$ of $S$ over $A$. 
By \cite[Th\'{e}or\`{e}me 12.2.4]{EGAIV3} and the argument in the proof of \cite[Proposition 4.3.11(c)]{HH_char0}, there exists a dense open subset $U \subseteq \Spec A$ such that $R_{\mu}$ and $S_{\mu}$ are both normal and $R_{\mu}$ is $\Q$-Gorenstein for all closed points $\mu \in U$. 
It then follows from \cite[Corollary 1.2]{Taylor20} that $s(R_{\mu}) \ge s(S_{\mu})$ for all closed points $\mu \in U$, which implies the assertion. 
\end{enumerate}
\end{proof}

Rings of strongly $F$-regular type conjecturally coincide with singularities of klt type (see, for example,  \cite[Section 4.6]{Hara-Watanabe}). 
Therefore, assuming this conjecture, we see from Proposition \ref{prop F-signature in characteristic 0} (1) that if $s(R)>0$, then $\Spec R$ is of klt type. 
Carvajal-Rojas, Schwede and Tucker ask in \cite[Question 5.9]{CRST} whether the converse holds. 


\begin{conjAd}[\textup{cf.~\cite[Question 5.9]{CRST}}]
Let $(R, \m)$ be a $d$-dimensional local domain essentially of finite type over $\C$. If $\Spec R$ is of klt type, then $s(R)>0$. 
\end{conjAd}

We simply say that Conjecture A holds if Conjecture $\textup{A}_d$ holds for all positive integers $d$. 

\begin{prop}
Conjecture $\textup{A}_2$ holds when $R$ has residue field isomorphic to $\C$. 
\end{prop}
\begin{proof}
First, note that two-dimensional singularities of klt type are $\Q$-Gorenstein. 
Therefore, taking the index one cover of $\Spec R$, we may assume by \cite[Theorem 3.1]{CRST} that $R$ is Gorenstein, that is, $\Spec R$ is a rational double point. 

Let $x \in X$ be a closed point of a complex algebraic variety $X$ such that $\sO_{X, x} \cong R$. 
Suppose that $(x_A \in X_A)$ is a model of $(x \in X)$ over a finitely generated $\Z$-subalgebra $A$ of $\C$.  
Let $(x_{\overline{\mu}} \in X_{\overline{\mu}})$ be the base change of $(x_{\mu} \in X_{\mu})$ to an algebraic closure $\overline{\kappa(\mu)}$ of the base field $\kappa(\mu)$. 
Then, after enlarging $A$ if necessary, we may assume that $(x_{\overline{\mu}} \in X_{\overline{\mu}})$ is a rational double point of the same type as $\Spec R$ for all closed points $\mu \in \Spec A$. 
Since the $F$-signature of $(x_{\overline{\mu}} \in X_{\overline{\mu}})$ depends only on its $ADE$ type and not on the characteristic of the residue field (if it is greater than 5) by \cite[Example 18]{HL02}, we may further assume, after possibly enlarging $A$ again, that there exists a constant $C>0$ such that $s(\sO_{X_{\overline{\mu}}, x_{\overline{\mu}}})>C$ for all $\mu$. 
By \cite[Theorem 5.6]{Yao06}, this implies that $s(\sO_{X_{\mu}, x_{\mu}})>C$ for all $\mu$,  
and therefore, $s(R)>C$. 
\end{proof}

\begin{rem}\label{toric remark}
It follows from \cite[Theorem 1.1]{Singh05} that Conjecture $A$ holds for localizations of toric rings. 
This case is also covered by our main result Theorem \ref{Uniform positivity of F-signature of pure subrings}. 
\end{rem}

The following is the main result of this section, which says that the uniform positivity of $F$-signature descends under pure morphisms. 

\begin{thm}\label{Uniform positivity of F-signature of pure subrings} 
Let $(R, \m) \to (S, \n)$ be a pure local $\C$-algebra homomorphism between
normal local rings essentially of finite type over $\C$.  
If $\Spec S$ is of klt type and $s(S)>0$, then $s(R)>0$. 
\end{thm}

The proof of Theorem \ref{Uniform positivity of F-signature of pure subrings} is deferred to near the end of this section, following Proposition \ref{local alpha invariant under pure morphisms}.

We apply Theorem \ref{Uniform positivity of F-signature of pure subrings} to the case of reductive quotient maps. 
A \textit{reductive quotient singularity} over $\C$ is the spectrum $\Spec R$ of a local $\C$-algebra $(R,\m)$ 
such that its $\m$-adic completion $\widehat{R}$ is isomorphic to the ring of invariants $\C[[x_1, \dots, x_n]]^G$ where $G$ is a linearly reductive group over $\C$ acting on a formal power series ring $\C[[x_1, \dots, x_n]]$. 

The following fact is well known to experts, but we include it here due to the lack of a reference. 
\begin{lem}\label{linearization}
    Let $k$ be an algebraically closed field, $S$ be a formal power series ring $k[[x_1,\dots,x_n]]$ over $k$ and $G$ be a linearly reductive group acting on the $k$-algebra $S$. Then the action of $G$ on $S$ can be linearized, i.e., there exist a $k$-algebra automorphism $\theta$ on $S$ and a group homomorphism $\sigma:G\to \operatorname{GL}_n(k)$ such that the diagram
    \[
        \xymatrix{
            S \ar[d]_{\theta} \ar[r]^{\sigma(g) \cdot} & S \ar[d]^{\theta} \\
            S \ar[r]_{g\cdot} & S
        } 
    \]
    commutes for all $g\in G$, where $\operatorname{GL}_n(k)$ acts naturally on $S$.
\end{lem}
\begin{proof}
This follows from an argument similar to that in \cite[Proof of Corollary 1.8]{Satriano}. Let $\n$ be the maximal ideal of $S$. 
Since $G$ acts on $\n/\n^2$, we have an $n$-dimensional $G$-representation $\sigma:G\to \operatorname{GL}_n(k)$. 
As $G$ is linearly reductive, for every $i \ge 2$, there exists a $G$-equivariant $k$-linear section $s_i:\n/\n^i\to \n/\n^{i+1}$ of the natural surjection $\n/\n^{i+1}\twoheadrightarrow \n/\n^{i}$. 
These maps induce a $G$-equivariant $k$-linear section $s:\n/\n^2\to \n$ of the natural surjection $\n\to \n/\n^2$. 
Define a $k$-algebra endomorphism $\theta:S\to S$ by $x_i\mapsto s(x_i+\n^2)$ for each $i$. Then $\theta$ is an automorphism because $s(x_i+\n^2)+\n^2=x_i+\n^2$.
\end{proof}

As an immediate corollary of Theorem \ref{Uniform positivity of F-signature of pure subrings}, we prove that Conjecture $A$ holds for reductive quotient singularities over $\C$. 
\begin{cor}\label{reductive quotient}
	Let $(R,\m)$ be a local domain essentially of finite type over $\C$. 
    If $\Spec R$ is a reductive quotient singularity over $\C$, then $s(R)>0$. 
\end{cor}
\begin{proof}
We use the following fact:

\begin{cl}
Let $(R,\m)$ be a local domain essentially of finite type over $\C$. 
If its $\m$-adic completion $\widehat{R}$ is a pure subring of a regular local ring $(B,\m_B)$, then $R$ itself is a pure subring of some regular local ring essentially of finite type over $\C$.    
\end{cl}
\begin{proof}[Proof of Claim]
It follows from Popescu's theorem that $B$ can be written as a direct limit $\varinjlim_{\lambda \in \Lambda} B_{\lambda}$ of smooth $\C$-algebras $B_\lambda$. 
Let $A$ be a domain of finite type over $\C$ and let $\mathfrak{p} \subset A$ be a prime ideal such that $R\cong A_{\mathfrak{p}}$. 
Then the natural injection $A \to B$ factors through $B_{\lambda_0}$ for some $\lambda_0 \in \Lambda$ (see \cite[Tag 07C3]{stacks-project}). 
Letting $S$ be the localization of $B_{\lambda_0}$ at the prime ideal $\m_B \cap B_{\lambda_0}$, we see that $R\to S$ is a pure local homomorphism.
\end{proof}

Suppose that $R$ is a reductive quotient singularity over $\C$. 
By Lemma \ref{linearization}, the completion $\widehat{R}$ of $R$ is a pure subring of a formal power series ring $\C[[x_1, \dots, x_n]]$. 
Applying the above claim to our $R$, we see that $R$ is a pure subring of a regular local ring essentially of finite type over $\C$.  
Therefore, it follows from Proposition \ref{prop F-signature in characteristic 0} and Theorem \ref{Uniform positivity of F-signature of pure subrings} that $s(R)>0$. 

\end{proof}

In order to show lemmas necessary for the proof of Theorem \ref{Uniform positivity of F-signature of pure subrings}, we work with the following setting. 

\begin{setting}\label{setup of pure morphisms}
	Let $(R,\m) \hookrightarrow (S,\n)$ be a pure local $\C$-algebra homomorphism between normal local rings essentially of finite type over $\C$, and let $d$ denote the dimension of $R$. 
	Fix a non-principal ultrafilter $\mathcal{F}$ on the set $\mathcal{P}$ of prime numbers and a field isomorphism $\gamma \colon \ulim_{p} \overline{\mathbb{F}_p}\xrightarrow{\sim} \C$. 
	Let $(R_p)_{p \in \mathcal{P}}$ and $(S_p)_{p \in \mathcal{P}}$ be approximations of $R$ and $S$, respectively. 
\end{setting}

As a key ingredient in the proof of Theorem \ref{Uniform positivity of F-signature of pure subrings}, we introduce a generalization of ultra-$F$-regularity, originally introduced by Schoutens \cite{Schoutens05}. 

\begin{defn}\label{ultra Freg def}
    With notation as in Situation \ref{setup of pure morphisms}, 
    let $f$ be a nonzero element of $R_{\infty}$ and $t_p>0$ be an rational number for almost all $p$, and set $t:=\ulim_p t_p$. 
    We then say that the pair $(R, f^t)$ is \textit{ultra-$F$-regular} if for every nonzero element $c \in R_{\infty}$, there exists a non-standard natural number $\epsilon\in {^*\N}$ such that the $R$-module homomorphism 
	\[
		R\to F^\epsilon_* R_\infty; \quad x\mapsto F^\epsilon_*\left(cf^{\lceil t\pi^\epsilon \rceil}F^\epsilon(x)\right)
	\]
	is pure.
\end{defn}

\begin{rem}
With notation as in Definition \ref{ultra Freg def}, suppose that $f$ is a nonzero element of $R$ and almost all $t_p$ are equal to some rational number $t_0$. Then $t$ can be identified with $t_0$, and 
$(R, f^t)$ is ultra-$F$-regular in the sense of Definition \ref{ultra Freg def} if and only if the pair $(R, f^{t_0})$ is ultra-$F$-regular in the sense of \cite[Definition 5.6]{Yam23a}. 
\end{rem}

The following proposition explains the relationship between the ultra-$F$-regularity of a local ring essentially of finite type over $\C$ and the strong $F$-regularity of its approximations. 
\begin{prop}[{cf.~\cite[Theorem 3.12]{Yamaguchi25}}] \label{F-regular for almost all p -> ultra-F-regular}\label{ultra-F-regular <-> F-regular type}
    With notation as in Situation \ref{setup of pure morphisms}, let $f_p$ be a nonzero element of $R_p$ and let $t_p$ be a positive rational number for almost all $p$, and set $f=\ulim_p f_p$ and $t=\ulim_p t_p$. 
\begin{enumerate}[label=$(\arabic*)$]
\item If the pair $(R_p, t_p \Div(f_p))$ is strongly $F$-regular for almost all $p$, then $(R, f^t)$ is ultra-$F$-regular. 
\item Suppose that the anti-canonical ring $\mathscr{R}:=\bigoplus_{i\ge 0}R(-iK_X)$ of $X:=\Spec R$ is finitely generated. 
Then the pair $(R_p, t_p \Div(f_p))$ is strongly $F$-regular for almost all $p$ if and only if $(R, f^t)$ is ultra-$F$-regular. 
\end{enumerate}  
\end{prop}
\begin{proof}
 (1) Let $\eta$ be a socle generator for $H^d_\m(\omega_R)$ and let $(\eta_p)_{p \in \mathcal{P}}$ be its approximation. Here note that a generator of the socle is unique up to multiplication by a unit. Then $\eta_p$ is a socle generator for $H^d_{\m_p}(\omega_{R_p})$ for almost all $p$, since the natural map $H_\m^d(\omega_R)\to \ulim_pH^d_{\m_p}(\omega_{R_p})$ is injective (see \cite[Propositions 3.3, 3.5 and the proof of Proposition 3.9]{Yamaguchi25a}). 
 Fix a nonzero element $c=\ulim_p c_p$ of $R_\infty$. 
 For almost all $p$, the pair $(R_p,t_p\Div(f_p))$ is strongly $F$-regular, so there exists an integer $e_p\ge 1$ such that $\eta_p\otimes F^{e_p}_*(c_pf_p^{\lceil tp^{e_p} \rceil})\neq 0$ in $H^d_{\m_p}(\omega_{R_p})\otimes_{R_p}F^{e_p}_*R_p$.  
 Letting $\epsilon:=\ulim_p e_p$, we have the following commutative diagram: 
    \[
        \xymatrix{
            H^d_{\m}(\omega_R) \ar[d] \ar[r] & H^d_\m(\omega_R)\otimes_R F^{\epsilon}_*R_\infty \ar[d]\\
            \ulim_p H^d_{\m_p}(\omega_{R_p}) \ar[r] & \ulim_p (H^d_{\m_p}(\omega_{R_p})\otimes_{R_p} F^{e_p}_*R_p).
        }
    \]
    Here the top horizontal map is given by $\xi\mapsto \xi\otimes F^\epsilon_*(cf^{\lceil t\pi^\epsilon \rceil})$ and the bottom horizontal map is given by $\ulim_p\xi_p\mapsto \ulim_p (\xi_p\otimes F^{e_p}_*(c_pf_p^{\lceil tp^{e_p}\rceil}))$. 
    Under the right vertical map, we have   
\[
\eta\otimes F^\epsilon_*(cf^{\lceil t\pi^\epsilon \rceil}) \mapsto \ulim_p (\eta_p\otimes F^{e_p}_*(c_pf_p^{\lceil tp^{e_p} \rceil})).
\]
Since the image is nonzero, the element $\eta\otimes F^\epsilon_*(cf^{\lceil t\pi^\epsilon \rceil})$ is nonzero in $H^d_\m(\omega_R)\otimes_R F^\epsilon_*R_\infty$. 
    It follows that the map $R\to F^\epsilon_*R_{\infty}$ sending $1$ to $F^\epsilon_*(cf^{\lceil t\pi^\epsilon \rceil})$ is pure. 

(2) We prove the ``if" part. We may assume that $K_X$ is an effective Weil divisor on $X$, so that $R(-iK_X)$ is an ideal of $R$ for all $i \ge 0$.  
Let $\mathscr{M}:=\m+\mathscr{R}_{>0}$ be the unique homogeneous maximal ideal of $\mathscr{R}$. First, we show the following claim.
    \begin{cl}
        If $(R,f^t)$ is ultra-$F$-regular, then so is $(\mathscr{R}_{\mathscr{M}},f^t)$.
    \end{cl}
    \begin{clproof}
     The case $\m = 0$ follows from (1), so assume $\m \neq 0$. We identify $H^{d+1}_{\mathscr{M}}(\omega_\mathscr{R})$ with $\bigoplus_{i>0}H^d_\m(R(iK_X))$ by \cite[Theorem 2.2]{Watanabe94}.  
    Note also that 
    $\Soc_R H^d_{\m}(\omega_R)=\Soc_{\mathscr{R}}H^{d+1}_{\mathscr{M}}(\omega_{\mathscr{R}})$. 
    It suffices to consider $0\neq c=\ulim_p c_p \in \m \mathscr{R}_\infty$.
    For almost all $p$, choose $i_p \ge 0$ such that the $i_p$-th homogeneous part $c_{p,i_p}$ of $c_p$ is nonzero. Set 
    \[
    i:=\ulim_p i_p, \quad 
    c_i:=\ulim_p c_{p,i_p}\in \m R_\infty\setminus\{0\}.
    \]
    Since $(R,f^t)$ is ultra-$F$-regular, there exists a non-standard natural number $\epsilon\in {^*\N}$ such that the $R$-linear map 
    \[F^i(c) f^{\lceil t\pi^\epsilon\rceil} F^{\varepsilon} \colon R\to F^\epsilon_*R_\infty; \quad  1\mapsto F^\epsilon_*(F^i(c_i) f^{\lceil t\pi^\epsilon\rceil})\] 
    is pure. Necessarily, $\epsilon>i$; otherwise, this map would become zero after tensoring with $R/\m$. It factors through $c_if^{\lceil t\pi^\epsilon \rceil}F^\epsilon$, so the latter map is also pure.
    On the other hand, it follows from \cite[Proposition 3.8]{Yamaguchi25} that there exists an $R$-linear map 
    \[\psi:H^{d+1}_{\mathscr{M}}(\omega_{\mathscr{R}})\otimes_{\mathscr{R}}F^{\epsilon}_*(\mathscr{R}_{\mathscr{M}})_{\infty}\to H^d_{\m}(\omega_R)\otimes_RF^{\epsilon}_*R_\infty\] 
    sending $\xi\otimes F^\epsilon_*(cg)$ to $\xi\otimes F^\epsilon_*(c_i g)$ for each $\xi\in H_\m^d(\omega_R)\subseteq H^{d+1}_{\mathscr{M}}(\omega_{\mathscr{R}})$ and $g\in R_\infty$. 
    Therefore, we have the following commutative diagram: 
         \[
            \xymatrix{
              H^{d+1}_{\mathscr{M}}(\omega_{\mathscr{R}}) \ar[r] & H^{d+1}_{\mathscr{M}}(\omega_{\mathscr{R}})\otimes_{\mathscr{R}}F^\epsilon_*(\mathscr{R}_\mathscr{M})_\infty \ar[d]^{\psi}\\
             H^d_\m(\omega_R) \ar@{^{(}->}[u] \ar[r]^-{1 \otimes c_if^{\lceil t\pi^\epsilon\rceil} F^{\varepsilon}}& H^d_{\m}(\omega_R)\otimes_R F^\epsilon_*R_\infty,
            }
         \]
         where the second horizontal map sends $\xi$ to $\xi\otimes F^\epsilon_*(cf^{\lceil t\pi^\epsilon \rceil})$ and the diagonal arrow sends $\eta$ to $\eta \otimes F^\epsilon_*(c_if^{\lceil t\pi^\epsilon\rceil})$. 
         The diagonal arrow is injective by the purity of $c_if^{\lceil t\pi^\epsilon\rceil} F^{\varepsilon}$,  and the socle of $H^{d+1}_{\mathscr{M}}(\omega_{\mathscr{R}})$ coincides with that of $H^d_{\m}(\omega_R)$, so the second horizontal map is also injective. 
         Thus, $(\mathscr{R}_{\mathscr{M}},f^t)$ is ultra-$F$-regular.
    \end{clproof}
    
    Suppose that $(R,f^t)$ is ultra-$F$-regular. 
    By the above claim, $(\mathscr{R}_\mathscr{M},f^t)$ is also ultra-$F$-regular. Suppose that $((\mathscr{R}_{\mathscr{M}})_p,t_p\Div(f_p))$ is not strongly $F$-regular for almost all $p$. Then, for almost all $p$, there exists a nonzero element $c_p\in \mathscr{R}_p$ such that for any $e\ge 0$, $c_pf_p^{\lceil tp^{e}\rceil}F^{e}:(\mathscr{R}_{\mathscr{M}})_p \to F^{e}_*(\mathscr{R}_{\mathscr{M}})_p$ is not pure. Let $c:=\ulim_p c_p$.
    Since $\mathscr{R}$ is quasi-Gorenstein by \cite[Theorem 4.5]{GHNV90}, we have
    \[
        H^{d+1}_{\mathscr{M}}(\omega_{\mathscr{R}})\otimes_{\mathscr{R}}M \cong H^{d+1}_{\mathscr{M}}(\mathscr{R})\otimes_{\mathscr{R}}M \cong H^{d+1}_{\mathscr{M}}(M)
    \]
    for every $\mathscr{R}$-module $M$, where the second isomorphism follows from the fact the $(\dim \mathscr{R})$-th local cohomology module is computed as the rightmost cohomology of the Čech complex associated to $\mathscr{R}$. 
    Thus, for each $\epsilon\in {^*\N}$, there exists the following commutative diagram:
\[
\xymatrix@C=10pt{
            H_{\mathscr{M}}^{d+1}(\omega_{\mathscr{R}}) \ar[r]^-{1\otimes cf^{\lceil t \pi^\epsilon\rceil}F^\epsilon} \ar[d]& H_{\mathscr{M}}^{d+1}(\omega_{\mathscr{R}})\otimes_{\mathscr{R}}F^\epsilon_*(\mathscr{R}_\mathscr{M})_{\infty}\ar[d] \ar[r]^-{\cong} &H_{\mathscr{M}}^{d+1}(F^\epsilon_*(\mathscr{R}_\mathscr{M})_{\infty}) \ar[d]&\\
            \ulim_{p}H_{\mathscr{M}_p}^{d+1}(\omega_{\mathscr{R}_p}) \ar[r] &\ulim_p (H_{\mathscr{M}_p}^{d+1}(\omega_{\mathscr{R}_p})\otimes_{\mathscr{R}_p} F^{e_p}_*(\mathscr{R}_{\mathscr{M}})_p) \ar[r]^-{\cong}& \ulim_pH_{\mathscr{M}_p}^{d+1}(F^{e_p}_*(\mathscr{R}_\mathscr{M})_{p}).
        }
\] 
    Since the right vertical map is injective by \cite[Proposition 3.9]{Yamaguchi25a}, $1\otimes cf^{\lceil t \pi^\epsilon\rceil}F^\epsilon$ is not injective, a contradiction. 
    Hence, $((\mathscr{R}_{\mathscr{M}})_p,t_p\Div(f_p))$ is strongly $F$-regular for almost all $p$. 
    For almost all $p$, the ring $(\mathscr{R}_{\mathscr{M}})_p$ is the localization at the unique homogeneous maximal ideal of an $\N$-graded $R_p$-algebra whose degree-zero component is isomorphic to $R_p$, so the natural homomorphism $R_p\to (\mathscr{R}_\mathscr{M})_{p}$ is pure. 
    Therefore, $(R_p,t_p\Div(f_p))$ is strongly $F$-regular for almost all $p$.
\end{proof}

The authors learned the proof of the following lemma from Kenta Sato, who they thank.

\begin{lem} \label{normalized order under pure mophisms}
With notation as in Theorem \ref{Uniform positivity of F-signature of pure subrings}, suppose that $(R_A,\m_A, S_A, \n_A)$ is a model of $(R, \m, S, \n)$ over a finitely generated $\Z$-subalgebra $A \subseteq \C$. 
Then there exist a constant $C>0$ and a dense open subset $U \subseteq \Spec A$ such that 
\[
C \cdot \overline{\mathrm{ord}}_{\m_{\mu}}(f) \ge \overline{\mathrm{ord}}_{\n_{\mu}}(f)
\]
for all closed points $\mu \in U$ and all nonzero elements $f \in R_{\mu}$. 
\end{lem}
\begin{proof}
By \cite[Corollary 2]{Huebl01}, there exists a divisorial valuation $v$ of the fractional field $\mathrm{Frac}(R)$ of $R$, centered on $\m$, that extends to a divisorial valuation $w$ of the fractional field $\mathrm{Frac}(S)$ of $S$, centered on $\n$. 
Enlarging $A$ if necessary, we may take a model $(v_A, w_A)$ of $(v, w)$ over $A$. 
Since $w_{\mu}$ is a divisorial valuation centered on $\n_{\mu}$ for general closed points $\mu \in \Spec A$, 
we have $w_{\mu}(\overline{\n_{\mu}^n}) \ge n$ for every integer $n \ge 1$. 
In particular, 
\[v_{\mu}(f)=w_{\mu}(f) \ge \overline{\mathrm{ord}}_{\n_{\mu}}(f)\]
for all nonzero elements $f \in R_{\mu}$ and such $\mu$. 

Let $\pi:Y \to X=\Spec R$ be a log resolution of $\m$ with exceptional divisor $E=\sum_{i=1}^r E_i$ such that $v$ corresponds to $E_1$ and $\m \sO_Y=\sO_Y(-\sum_{i=1}^ra_i E_i)$. 
We also fix a $\pi$-ample Cartier divisor $H$ on $Y$. 
Enlarging $A$ again, we may take a model $(\pi_A:Y_A \to X_A, E_A=\sum_{i=1}^r E_{i,A}, H_A)$ of $(\pi:Y \to X, E=\sum_{i=1}^r E_i, H)$ over $A$. 
Then for general closed points $\mu \in \Spec A$, the morphism $\pi_{\mu}:Y_{\mu} \to X_{\mu}=\Spec R_{\mu}$ is a log resolution of $\m_{\mu}$ with exceptional divisor $E_{\mu}=\sum_{i=1}^r E_{i, \mu}$ such that $v_{\mu}$ corresponds to $E_{1, \mu}$ and $\m_{\mu} \sO_{Y,\mu}=\sO_{Y,\mu}(-\sum_{i=1}^r a_i E_{i,\mu})$. Moreover, $H_{\mu}$ is a $\pi_{\mu}$-ample Cartier divisor on $Y_{\mu}$. 

We now apply Izumi's theorem. 
Using essentially the same argument as in the proof of Izumi's theorem given in \cite{BFJ14}\footnote{The base field is assumed to have characteristic zero in \cite{BFJ14}, but this is only used to guarantee the existence of log resolutions. Since we already have a log resolution $\pi_{\mu}$, the argument applies in our setting as well.}, there exists a constant $C_{\mu}>0$ such that 
\[C_{\mu} \cdot \overline{\mathrm{ord}}_{\m_{\mu}}(f) \ge v_{\mu}(f)\]
for all nonzero elements $f \in R_{\mu}$. 
Note that $C_{\mu}$ is determined by the intersection numbers $(H_{\mu}^{d-2} \cdot E_{i,\mu} \cdot E_{j, \mu})$, as well as $a_1$ and $r$.
Since these intersection numbers agree with $(H^{d-2} \cdot E_i \cdot E_j)$ for general closed points $\mu \in \Spec A$, 
we can choose the same constant $C$ (independent of $\mu$) for such $\mu$. 
This completes the proof. 

\end{proof}

\begin{prop} \label{local alpha invariant under pure morphisms}
With notation as in Situation \ref{setup of pure morphisms}, suppose that the anti-canonical ring $\mathscr{R}=\bigoplus_{i\ge 0}R(-iK_X)$ of $X:=\Spec R$ is finitely generated.
Then there exists a constant $C>0$ such that $C \cdot \alpha_F(R_p)\ge \alpha_F(S_p)$ for almost all $p$.
\end{prop}
\begin{proof}
Choose $C>0$ to be a constant satisfying the inequality in Lemma \ref{normalized order under pure mophisms}.  
By Remark \ref{approximation vs reduction}, together with the fact that normalized orders do not change under faithfully flat extensions (see, for example, \cite[Proposition 1.6.2]{Huneke-Swanson}), we have 
\[
C \cdot \overline{\mathrm{ord}}_{\m_p}(f) \ge \overline{\mathrm{ord}}_{\n_p}(f)
\]
for almost all $p$. 

Suppose to the contrary that for almost all $p$, the inequality $C \cdot \alpha_F(R_p)<\alpha_F(S_p)$ holds.
Then, for each such $p$, there exists a nonzero element $f_p\in \m_p$ such that
\[
C \cdot \operatorname{fpt}_{R_p}(f_p) \overline{\mathrm{ord}}_{\m_p}(f_p)<  \alpha_F(S_p)
\le \operatorname{fpt}_{S_p}(f_p)\overline{\mathrm{ord}}_{\n_p}(f_p). 
\] 
By the choice of $C$, we deduce that $\operatorname{fpt}_{R_p}(f_p)<\operatorname{fpt}_{S_p}(f_p)$ for almost all $p$. 
Choose a rational number $t_p$ lying in the interval $(\operatorname{fpt}_{R_p}(f_p), \operatorname{fpt}_{S_p}(f_p))$. 
Then the pair $(S_p, t_p \Div(f_p))$ is strongly $F$-regular for almost all $p$, which implies, by Proposition \ref{ultra-F-regular <-> F-regular type}, that $(S, f^t)$ is ultra-$F$-regular, where $f=\ulim_p f_p$ and $t=\ulim_p t_p$. 
Therefore, for every nonzero element $c \in R_\infty\subseteq S_\infty$, there exists a non-standard natural number $\varepsilon \in {}^*\N$ such that the $S$-linear map $S \to F^{\varepsilon}_*S_{\infty}$ sending $x$ to $F^{\varepsilon}_*(cf^{\lceil t\pi^{\epsilon}\rceil}F^{\epsilon}(x))$ is pure. 
Now we consider the following commutative diagram: 
\[
		\xymatrix{
		R \ar[r] \ar[d]& F^\epsilon_*R_\infty \ar[d] \\
		S \ar[r]& F^\epsilon_* S_\infty,
		}
\]
where the horizontal maps send $x$ to $F^\epsilon_*\left(cf^{\lceil t\pi^\epsilon \rceil}F^\epsilon(x)\right)$. 
Since the inclusion map $R \to S$ is pure, the top horizontal map is also pure, which means that $(R, f^t)$ is ultra-$F$-regular. 
Applying Proposition \ref{ultra-F-regular <-> F-regular type} again, 
we conclude that $(R_p,t_p \Div(f_p))$ is strongly $F$-regular for almost all $p$. 
However, this contradicts the choice of $t_p$. 
\end{proof}

We are now ready to prove the main result of this section, Theorem \ref{Uniform positivity of F-signature of pure subrings}. 
\begin{proof}[Proof of Theorem \ref{Uniform positivity of F-signature of pure subrings}]
First note that $\Spec R$ is also of klt type by a result of Zhuang \cite[Theorem 1.1]{Zhuang24}. Hence, the anti-canonical ring of $R$ is finitely generated by \cite{BCHM}. 
Suppose that $(R_A, \m_A)$ is a model of $(R,\m)$ over a finitely generated $\Z$-subalgebra $A \subseteq \C$. 
Assume to the contrary that $s(R)=0$, that is, for every constant $C>0$, there exists a dense subset $U_C$ of closed points of $\Spec A$ such that 
$s(R_{\mu}) \le C$ for all $\mu \in U_C$. 
Then, by an argument similar to the proof of \cite[Proposition 4.6]{Yamaguchi25a}, we can construct a dense subset $U=\{\mu_i\}_{i=1}^{\infty}$ of closed points of $\Spec A$, a non-principal ultrafilter $\mathcal{F}$ on $\mathcal{P}$ and a field isomorphism $\gamma: \ulim_{p}\bar{\F_p} \xrightarrow{\sim} \C$ satisfying the following four conditions:
\begin{enumerate}[label=(\roman*)]
\item $p_i<p_{i+1}$ for all $i \ge 1$,
\item $s(R_{\mu_i}) \le 1/i$ for all $i\ge 1$,
\item $\{p_i\}_{i=1}^{\infty}\in \mathcal{F}$,
\item a natural map $R_{\mu_i} \to R_{p_i}$ is flat local and $\m_{{\mu_i}} R_{p_i}=\m_{p_i}$, 
\end{enumerate}
where $p_i$ is the characteristic of the residue field $A/\mu_i$ and $((R_{p}, \m_p))_{p \in \mathcal{P}}$ is an approximation of the local ring $(R,\m)$ with respect to $\mathcal{F}$ and $\gamma$. 

\begin{cl}
Let $((S_{p}, \n_p))_{p \in \mathcal{P}}$ be an approximation of $(S, {\n})$ with respect to $\mathcal{F}$ and $\gamma$. 
Then there exists a constant $C_0>0$ such that $s(S_p)>C_0$ for almost all $p$. 
\end{cl}
\begin{proof}[Proof of Claim]
Enlarging $A$ if necessary, choose a model $(S_A, \n_A)$ of $(S, \n)$ over $A$. 
Fix a constant $C_0$ satisfying $0<C_0<s(S)$. 
Then there exists a dense open subset $V \subseteq \Spec A$ such that $s(S_{\mu})>C_0$ for all closed points $\mu\in V$. 
For almost all $p$, by Remark \ref{approximation vs reduction}, one can find a closed point $\mu_p\in V$ and a flat local homomorphism $S_{\mu_p}\to S_p$ with $\n_{{\mu_p}} S_p=\n_p$. 
It then follows from \cite[Theorem 5.6]{Yao06} that  $s(S_p)=s(S_{\mu_p})>C_0$ for almost all $p$.
\end{proof}

Combining the above claim with Proposition \ref{comparison of F-signature and local alpha invariant}, we find a constant $C_1>0$ such that $\alpha_F(S_p)>C_1$ for almost all $p$. Note that $N$ in Proposition \ref{comparison of F-signature and local alpha invariant} (2) depends only on the constant appearing in Izumi's theorem, which, by the proof of Lemma \ref{normalized order under pure mophisms}, does not depend on $p$.
Then Proposition \ref{local alpha invariant under pure morphisms} yields a constant $C_2>0$ such that $\alpha_F(R_p)>C_2$ for almost all $p$.
Applying Proposition \ref{comparison of F-signature and local alpha invariant}, we obtain a constant $C_3>0$ such that $s(R_p)>C_3$ for almost all $p$.  
Finally, by \cite[Theorem 5.6]{Yao06} again, we have 
	\[
		0<C_3 \le s(R_{p_i}) = s(R_{\mu_i}) \leq \frac{1}{i}
	\]
	for infinitely many $i \ge 1$, which is a contradiction.
\end{proof}

We close this section with the following corollary concerning the Gorenstein case of Conjecture A. 
We say that Conjecture B holds if Conjecture A holds for Gorenstein rings. 
\begin{conjBd}
Let $(R, \m)$ be a $d$-dimensional Gorenstein local domain essentially of finite type over $\C$. If $\Spec R$ has klt singularities, then $s(R)>0$. 
\end{conjBd}
\begin{cor}\label{reduce to Gorenstein case}
For each integer $d \ge 2$, Conjecture $\textup{B}_{d+1}$ implies Conjecture $\textup{A}_{d}$. 
\end{cor}
\begin{proof}
Let $R$ be as in Conjecture $\textup{A}_{d}$.  
Let $\mathscr{R}:=\bigoplus_{i\ge 0} R(-iK_X)$ denote the anti-canonical ring of $X:=\Spec R$ and let $\mathscr{M}$ denote the unique homogeneous maximal ideal of $\mathscr{R}$. 
Since $X$ is of klt type, it follows from a result of \cite{BCHM} that $\mathscr{R}$ is a finitely generated $R$-algebra and $\Spec \mathscr{R}_{\mathscr{M}}$ is a Gorenstein klt singularity of dimension $d+1$. 
Therefore, by Conjecture $\textup{B}_{d+1}$, 
we have $s(\mathscr{R}_{\mathscr{M}})>0$ . 
Noting that the natural map $R\to \mathscr{R}_{\mathscr{M}}$ is a split local homomorphism, we see from Theorem \ref{Uniform positivity of F-signature of pure subrings} that $s(R)>0$. 
\end{proof}

\section{Local to global}
In this section, we discuss a global interpretation of Conjecture A introduced in the previous section. 

First, we introduce the notion of $F$-alpha invariants in equal characteristic zero, in a way similar to Definition \ref{F-signature in char0}, which defines $F$-signature in equal characteristic zero. 
\begin{defn}\label{global alpha in char 0}
Let $X$ be a projective variety over $\C$, let $\Delta$ be an effective $\Q$-Weil divisor on $X$ and let $\Gamma$ be a big $\Q$-Weil divisor on $X$. 
Suppose that $(X_A, \Delta_A, \Gamma_A)$ is a model of $(X, \Delta, \Gamma)$ over a finitely generated $\Z$-subalgebra $A$ of $\C$. Then $F$-alpha invariant $\alpha_F((X,\Delta);\Gamma)$ is defined as 
\[
\alpha_F((X,\Delta);\Gamma)=\sup_U \inf_{\mu \in U} \alpha_F((X_{\mu},\Delta_{\mu});\Gamma_{\mu}),
\]
where $U$ runs through all dense open subsets of $\Spec A$ and $\mu$ runs through all closed points of $U$. 
\end{defn}

\begin{conjCd}
Let $(X, \Delta)$ be a $d$-dimensional projective log Fano pair over $\C$. Then 
\[\alpha_F((X,\Delta);-(K_X+\Delta))>0.\] 
\end{conjCd}

We simply say that Conjecture C holds if Conjecture $\textup{C}_d$ holds for all positive
integers $d$.

\begin{thm}\label{local-global thm}
Conjecture C implies Conjecture B. 
Also, for each integer $d \ge 2$, Conjecture $\textup{C}_{d-1}$ implies Conjecture $\textup{B}_d$ for local rings $R$ with residue field $\C$. 
\end{thm}

In order to prove Theorem \ref{local-global thm}, we work with the following setting. 
\begin{setting}\label{set up}
Let $x$ be a closed point of an affine Gorenstein klt variety $X$ over $\C$.
Let $\sigma: Y \to X$ be a projective birational morphism that provides a Koll\'ar component $F$, that is, $\sigma$ is an isomorphism over $X\setminus\{x\}$ and $\sigma^{-1}(\{x\})$ is a prime divisor $F$ on $Y$ such that the pair $(Y,F)$ is plt and $-F$ is $\Q$-Cartier $\sigma$-ample. 
For each integer $j \ge 0$, we define the ideal $\ba_j \subseteq \sO_X$ as 
\[
\ba_j=\{f \in \sO_{X} \; | \; \mathrm{ord}_{F}(f)\ge j\}=H^0(Y,\sO_{Y}(-jF)) \subseteq \sO_X, 
\] 
and let $G$ denote the $\N$-graded $\C$-algebra $\bigoplus_{j \ge 0} \ba_j/\ba_{j+1}$. 
Since $-F$ is $\sigma$-ample, $G$ is finitely generated.  
\end{setting}

\begin{prop}\label{prop1}
With notation as in Situation \ref{set up}, 
suppose that $(x_A \in X_A, G_A)$ is a model of $(x \in X, G)$ over a finitely generated $\Z$-subalgebra $A \subseteq \C$. 
Then there exists a constant $C>0$ such that $\alpha_F(\sO_{X_{\mu},x_{\mu}}) \ge C \cdot \alpha_F\left(G_{\mu} \right)
$ for general closed points $\mu \in \Spec A$. 
\end{prop}

\begin{proof}
We define the $\Z$-graded $\sO_X$-algebra $\mathcal{R}'$ by 
\[\mathcal{R}'=\bigoplus_{j\in \Z} \ba_jt^{-j} \subseteq \sO_X[t, t^{-1}],\] 
where $\ba_{j}=H^0(Y, \sO_Y(-jF)) \subseteq \sO_X$. 
Since $F$ is $\sigma$-anti-ample, $\mathcal{R}'$ is finitely generated. 
The inclusion map $\C[t] \to \mathcal{R}'$ is flat, and we have the following isomorphisms. 
\begin{align*}
 \mathcal{R}'\otimes_{\C[t]}\C[t]/(t) & \cong G, \\ 
 \mathcal{R}'\otimes_{\C[t]}\C[t,t^{-1}] & \cong \sO_{X}[t,t^{-1}]. 
 \end{align*}
It follows from \cite[Section 2.4]{Li--Xu} (see also \cite[Proposition 2.10]{Liu--Zhuang}) that $G$ 
is isomorphic, as an $\N$-graded $\C$-algebra, to $\bigoplus_{j \ge 0} H^0(F, \sO_F(\lfloor -jF|_F \rfloor))$. 
Moreover, by \cite[Proposition 2.9]{Liu--Zhuang}, the spectrum of this graded ring has klt singularities.

Now, after enlarging $A$ if necessary, we may choose a model $\mathcal{R}'_A$ of $\mathcal{R}'$ over $A$. 
We may further assume that for every closed point $\mu \in \Spec A$, the ring $\mathcal{R}'_{\mu}$ is a finitely generated $\Z$-graded $\sO_{X_\mu}$-subalgebra of $\sO_{X_\mu}[t, t^{-1}]$\footnote{The fact that $\mathcal{R}'_{\mu}$ is a subalgebra of $\sO_{X_\mu}[t, t^{-1}]$ follows from \cite[Theorem 2.3.5(b)]{HH_char0}.}, and the inclusion map $\kappa(\mu)[t] \to \mathcal{R}'_{\mu}$ is flat.   
For each integer $j \in \Z$, let $\mathrm{pr}_j \colon \sO_{X_\mu}[t, t^{-1}] \to \sO_{X_\mu}$ denote the natural projection onto the degree $j$ component, and define $\ba_{j, \mu}$ to be the ideal of $\sO_{X_\mu}$ making the following commutative diagram commute: 
\[
\xymatrix{\mathcal{R}'_{\mu} \ar@{^{(}->}[r] \ar@{->>}[d]^{\mathrm{pr}_j} & \sO_{X_\mu}[t, t^{-1}] \ar@{->>}[d]^{\mathrm{pr}_j} \\
\ba_{j,\mu}t^{-j} \ar@{^{(}->}[r] & \sO_{X_\mu}t^{-j}.
}
\]
In particular, $\mathcal{R}'_\mu=\bigoplus_{j \in \Z} \ba_{j,\mu}t^{-j}$ and $G_{\mu}=\bigoplus_{j \ge 0} \ba_{j,\mu}/\ba_{j+1,\mu}$. 
Let 
\[T_{\mu}:=\mathcal{R}'_{\mu} \otimes_{\sO_{X_{\mu}}}  \sO_{X_{\mu}, x_\mu}=\bigoplus_{j \in \Z} \ba_{j,\mu} \sO_{X_{\mu}, x_\mu} t^{-j} \subseteq \sO_{X_{\mu}, x_\mu}[t, t^{-1}].\] 
Then $T_{\mu}$ is a ${}^*$local $\Z$-graded $\sO_{X_{\mu}, x_\mu}$-algebra. 
Note that for general closed points $\mu \in \Spec A$, we have the following isomorphisms: 
\begin{align*}
T_{\mu} \otimes_{\kappa(\mu)[t]}\kappa(\mu)[t]/(t) 
& \cong \left(\mathcal{R}'_{\mu}\otimes_{\kappa(\mu)[t]}\kappa(\mu)[t]/(t)\right) \otimes_{\sO_{X_{\mu}}} \sO_{X_{\mu}, x_\mu} \\
& \cong \left(\mathcal{R}'\otimes_{\C[t]}\C[t]/(t)\right)_{\mu} \otimes_{\sO_{X_{\mu}}} \sO_{X_{\mu}, x_\mu} \\
& \cong G_{\mu, x_{\mu}}, \\
T_{\mu} \otimes_{\kappa(\mu)[t]} \kappa(\mu)[t,t^{-1}] 
& \cong (\mathcal{R}'_{\mu} \otimes_{\kappa(\mu)[t]} \kappa(\mu)[t, t^{-1}])\otimes_{\sO_{X_{\mu}}} \sO_{X_{\mu}, x_\mu} \\
& \cong (\mathcal{R}' \otimes_{\C[t]} \C[t, t^{-1}])_{\mu} \otimes_{\sO_{X_{\mu}}} \sO_{X_{\mu}, x_\mu} \\
& \cong \sO_{X_{\mu}, x_\mu}[t,t^{-1}]. 
\end{align*}
Since $\Spec \left(\mathcal{R}'\otimes_{\C[t]}\C[t]/(t)\right)$ has klt singularities, 
we deduce from Theorem \ref{klt SFR correspondence} that there exists a dense open subset $U \subseteq \Spec A$ such that $G_{\mu, x_{\mu}}$
is strongly $F$-regular for all closed points $\mu \in U$. 

Fix a closed point $\mu \in U$. 
For a nonzero element $f\in \sO_{X_\mu, x_{\mu}}$, set
 \begin{align*}
 	\widetilde{f}&:=t^{-n}f\in \ba_{n,\mu} \sO_{X_{\mu}, x_\mu}t^{-n}\subseteq T_\mu, \\
 	\mathrm{in}(f)&:=[f]\in \ba_{n,\mu} \sO_{X_{\mu}, x_\mu}/\ba_{n+1,\mu} \sO_{X_{\mu},x_\mu} \subseteq G_{\mu, x_{\mu}}, 
\end{align*}
where $n$ is the largest integer $j$ with $f \in \ba_{j,\mu} \sO_{X_{\mu}, x_\mu}$. 
Let $s>0$ be a real number such that the pair $(G_{\mu, x_{\mu}}, \mathrm{in}(f)^s)$
is strongly $F$-regular, and in particular, $F$-rational.
It then follows from a combination of \cite[Propositions 6.15(1), 7.3]{Schwede-Takagi08} and Lemma \ref{graded Frational} that $(T_{\mu}, \widetilde{f}^s)$ is also $F$-rational. 
Since $\sO_{X_{\mu}, x_\mu}[t,t^{-1}]$ is a localization of $T_{\mu}$, the pair $(\sO_{X_{\mu}, x_\mu}[t,t^{-1}],f^s)$ is likewise $F$-rational. 
By Remark \ref{F-rational rem} and the fact that $\sO_{X_{\mu}, x_\mu}[t,t^{-1}]$ is Gorenstein, this implies that $(\sO_{X_{\mu}, x_\mu}[t,t^{-1}],f^s)$ is strongly $F$-regular. 
Finally, using the faithful flatness of the natural map $\sO_{X_{\mu}, x_\mu} \to \sO_{X_{\mu}, x_\mu}[t,t^{-1}]$, we conclude that $(\sO_{X_{\mu}, x_\mu},f^s)$ is strongly $F$-regular as well. 
Putting everything together, we obtain the inequality
 \[
 \operatorname{fpt}(\sO_{X_{\mu}, x_\mu},f)\ge \operatorname{fpt}\left(G_{\mu, x_{\mu}}, 
 \mathrm{in}(f)\right).
 \]

Meanwhile, let $\m_x$ be the maximal ideal of $\sO_X$ corresponding to the closed point $x$. 
By the uniform version of Izumi's theorem (see the proof of Lemma \ref{normalized order under pure mophisms}), there exists a constant $C>0$ such that for general closed points $\mu \in \Spec A$, we have $\ba_{n, \mu} \subseteq \m_{x_{\mu}}^{\lceil C n \rceil}$ for all integers $n \ge 1$. 
As a consequence, every nonzero element $f \in \sO_{X_{\mu}, x_{\mu}}$ satisfies $\mathrm{ord}_{\m_{x_{\mu}}}(f) \ge C \cdot \deg(\mathrm{in}(f))$, which leads to the inequality 
\[
\operatorname{fpt}(\sO_{X_{\mu}, x_{\mu}},f) \mathrm{ord}_{\m_{x_{\mu}}}(f) 
\ge C \cdot \operatorname{fpt}\left(G_{\mu, x_{\mu}}, 
\mathrm{in}(f)\right) \deg(\mathrm{in}(f)).
\]
Thus, $\alpha_F(\sO_{X_\mu, x_\mu}) \ge C \cdot \alpha_F(G_\mu)$ for general closed points $\mu \in \Spec A$.
\end{proof}

\begin{prop}\label{prop2}
With notation as in Situation \ref{set up}, let $a\ge 0$ be an integer such that $K_Y=\sigma^*K_X+aF$ and let $\Delta_F$ denote the $\Q$-Weil divisor on $F$ known as the different on $F$ $($see \cite[Definition 2.34]{Kol13} for its definition$)$. 
Suppose that $(\sigma_A\colon Y_A \to X_A, F_A, \Delta_{F,A}, G_A)$ is a model of $(\sigma \colon Y \to X, F, \Delta_F, G)$ over a finitely generated $\Z$-subalgebra $A \subseteq \C$. 
Then 
\[
\alpha_F\left(G_{\mu} \right)=(1+a) \cdot \alpha_F((F_\mu, \Delta_{F, \mu});-(K_{F_\mu}+\Delta_{F,\mu}))
\]
for general closed points $\mu \in \Spec A$. 
\end{prop}
\begin{proof}
 Given a $\Q$-Cartier $\Q$-Weil divisor $B$ on a normal variety $Z$, the Cartier index of $B$ at a point $z \in Z$ is denoted by $\mathrm{ind}(z \in Z, B)$. 
 Let $y$ be a codimension one point of $F$ and let $D$ be the prime divisor on $F$ corresponding to $y$. 
It follows from \cite[Theorem A.1]{HLS26} that 
\[
    \mathrm{ind}(y\in Y,F)=\mathrm{ind}(y\in F, F|_F).
\]
Therefore, if we set $m:=\mathrm{ind}(y\in Y,F)$, then there exists an integer $n$ such that $\mathrm{ord}_D(F|_F)=n/m$ and $\gcd(m,n)=1$. 
Applying this argument together with \cite[Proposition 16.6]{Kollaretal} we have prime divisors $D_i$ on $F$, positive integers $m_i$ and integers $n_i$ such that
\[
    \Delta_F=\sum_{i}\frac{m_i-1}{m_i}D_i, \quad F|_F=\sum_i \frac{n_i}{m_i}D_i,
\]
and $\gcd(m_i,n_i)=1$ for each $i$. 

We consider the following two graded rings: 
\begin{align*}
    S&:=\bigoplus_{i\ge 0}H^0(F,\sO_F(-ir(K_F+\Delta_F))),\\
    S'&:=\bigoplus_{i\ge 0} H^0(F,\sO_F(\lfloor -iF|_F \rfloor)), 
\end{align*}
where $r=\mathrm{ind}(K_F+\Delta_F)$. 
By \cite[Proposition 2.10]{Liu--Zhuang}, $S'$ is isomorphic to $G$ as a graded $\C$-algebra. 
Also, by the definition of $\Delta_F$, 
\[
K_F+\Delta_F=(K_Y+F)|_F=(\sigma^* K_X+(1+a)F)|_F=(1+a)F|_F.
\]
Therefore, $S$ is a Veronese subring of $S'$. 
Let 
$\pi:\Spec S'\to \Spec S$ denote the morphism induced by the inclusion and let $\Delta_S$ denote the $\Q$-Weil divisor on $S$ corresponding to $\Delta_F$.
It follows from \cite[(1.3) and Theorem 2.8]{Watanabe81} that $K_{S'}=\pi^*(K_S+\Delta_S)$. 

Now we reduce the setting modulo $p$. 
We may assume that the degree of the finite map $\pi_{\mu}: \Spec S'_{\mu} \to \Spec S_{\mu}$ is not divisible by the characteristic of $\kappa(\mu)$, so that the trace map $\mathrm{Tr}_{\pi_\mu}:\mathrm{Frac}(S'_{\mu}) \to \mathrm{Frac}(S_{\mu})$ is surjective, where $\mathrm{Frac}(S'_{\mu})$ (resp. $\mathrm{Frac}(S_{\mu})$) denotes the quotient field of $S'_{\mu}$ (resp. $S_{\mu}$).  
Since
\[\alpha_F((F_\mu, \Delta_{F, \mu});-(K_{F_\mu}+\Delta_{F,\mu}))=r \cdot \alpha_F(S_{\mu},\Delta_{S_{\mu}})\]
by Proposition \ref{characterization of alpha invariant using global fpt}, 
it suffices to prove that 
\[
\alpha_F(G_{\mu})=\alpha_F(S'_{\mu})=(1+a)r \cdot \alpha_F(S_{\mu},\Delta_{S_{\mu}})
\]
for general closed points $\mu \in \Spec A$. 
To prove this, it is enough to show that $\operatorname{fpt}(S'_{\mu},f)=\operatorname{fpt}((S_{\mu},\Delta_{S_{\mu}});f)$ for each nonzero nonunit homogeneous element $f \in S_{\mu}$.  
For every real number $t\ge 0$, \cite[Main Theorem]{Schwede-Tucker14} gives the equality 
\[ 
\mathrm{Tr}_{\pi_{\mu}}(\tau(S'_{\mu},f^t))=\tau((S_{\mu},\Delta_{S_{\mu}});f^t).
\]
By the surjectivity of the trace map $\mathrm{Tr}_{\pi_{\mu}}$, we have $\operatorname{fpt}(S'_{\mu},f)=\operatorname{fpt}((S_{\mu},\Delta_{S_{\mu}});f)$. 
\end{proof}

\begin{proof}[Proof of Theorem \ref{local-global thm}]
Let $(R, \m)$ be a Gorenstein local domain essentially of finite type over $\C$. 
Take a point $x$ on a Gorenstein algebraic variety $X$ over $\C$ such that $\sO_{X, x} \cong R$. 
By replacing $x$ with a closed point of the closure $\overline{\{x\}}$, we may assume that $x$ is a closed point of $X$. 
We now use the same notation as in Situation \ref{set up}. 
By Propositions \ref{prop1} and \ref{prop2} and Conjecture $\mathrm{C}$ (respectively $\mathrm{C}_{d-1}$), there exists a constant $\delta>0$ such that $\alpha_F(R_\mu)\ge \delta$ for general closed points $\mu$ of a suitable model, where $R_\mu=\sO_{X_\mu,x_\mu}$. Proposition \ref{comparison of F-signature and local alpha invariant} (1), then gives $s(R)>0$.
\end{proof}

In the latter half of this section, we consider two weaker versions of Conjecture $\mathrm{C}_d$: the first concerns varieties of Fano type, and the second concerns $\Q$-Fano varieties.

\begin{conjDd}\label{Fano type conjecture}
If $X$ is a $d$-dimensional complex variety of Fano type, then $\alpha_F(X)>0$. 
\end{conjDd}

\begin{conjEd}\label{Q-Fano conjecture}
If $X$ is a $d$-dimensional $\Q$-Fano variety over $\C$, then $\alpha_F(X)>0$. 
\end{conjEd}

We investigate the relationship between the above conjectures and those discussed in the previous sections.

\begin{prop}\label{easy implications}
The following statements hold.
\begin{enumerate}[label=$(\arabic*)$]
\item 
Conjecture $\mathrm{B}_{d+1}$ for local rings $R$ with residue field $\C$ implies Conjecture $\mathrm{E}_d$. 
\item 
Conjecture $\mathrm{C}_d$ implies Conjecture $\mathrm{D}_d$. 
\end{enumerate}
\end{prop}
\begin{proof}
(1) 
Let $X$ be a $d$-dimensional $\Q$-Fano variety over $\C$. 
By Lemma \ref{fpt of non homogeneous elements} and Proposition \ref{characterization of alpha invariant using global fpt}, applying Conjecture $\mathrm{B}_{d+1}$ to the localization of the anti-canonical ring $R(X, -K_X)$ of $X$ at its homogeneous maximal ideal yields $\alpha_F(X)>0$. 

(2) Let $(X,\Delta)$ be a log Fano pair over $\C$. 
It suffices to show that if 
\[
\alpha_F(X; -(K_X+\Delta))>0,
\]
then $\alpha_F(X)>0$. 
Since $-(n-1)K_X-n\Delta=-n(K_X+\Delta)-(-K_X)$ is $\Q$-effective for sufficiently large integers $n$, by Lemma \ref{effective divisor inequality}, we obtain 
\[
\alpha_F(X) \ge \alpha_F(X; -n(K_X+\Delta))=\frac{1}{n} \alpha_F(X; -(K_X+\Delta))>0.
\]
\end{proof}

\begin{lem}\label{effective divisor inequality}
Let $X$ be a globally $F$-regular projective variety over an $F$-finite field of characteristic $p>0$, let $\Gamma$ be a big $\Q$-Weil divisor on $X$ and let $E$ be a $\Q$-effective $\Q$-Weil divisor on $X$. Then $\alpha_F(X;\Gamma)\ge \alpha_F(X;\Gamma+E)$.
\end{lem}
\begin{proof}
Let $t$ be a real number greater than $\alpha_F(X;\Gamma)$. 
Then there exist effective $\Q$-Weil divisors $\Delta$ and $D$ on $X$ such that $\Delta \sim_\Q E$, $D\sim_\Q \Gamma$ and $t\ge \mathrm{gfst}(X,D)$. 
Then $D+\Delta \sim_\Q \Gamma+E$ and $t\ge \mathrm{gfst}(X,D)\ge \mathrm{gfst}(X,D+\Delta)$, 
which implies that $t\ge \alpha_F(X;\Gamma+E)$.
Therefore, $\alpha_F(X;\Gamma) \ge \alpha_F(X;\Gamma+E)$.
\end{proof}


The following theorem is the main result of the latter half of this section.

\begin{thm}\label{two conjectures are equivalent}
Conjecture $\mathrm{D}_d$ is equivalent to Conjecture $\mathrm{E}_d$. 
\end{thm}

To prove Theorem \ref{two conjectures are equivalent}, we first establish a couple of lemmas. 

\begin{lem}\label{F-alpha invariant under small birational map}
Suppose that $\pi:X \dashrightarrow Y$ is a small birational map between globally $F$-regular projective varieties over an $F$-finite field of characteristic $p>0$ and $\Gamma$ is a big $\Q$-Weil divisor on $X$. Then $\alpha_F(X;\Gamma)=\alpha_F(Y;\pi_*\Gamma)$.
\end{lem}
\begin{proof}
	This follows from an argument similar to \cite[Lemma 2.12]{GOST15}.
\end{proof}

\begin{lem}\label{F-alpha invariant under divisorial contraction}
Let $\pi:X\to Y$ be a birational morphism between $\Q$-Gorenstein globally $F$-regular projective varieties over an $F$-finite field of characteristic $p>0$. 
Suppose that $K_X-\pi^*K_Y\le 0$. Then 
\[
\alpha_F(X)\ge \min\{\alpha_F(Y),1\}.
\]
\end{lem}
\begin{proof}
Let $D$ be an effective $\Q$-Cartier $\Q$-Weil divisor on $X$ that is $\Q$-linearly equivalent to $-K_X$. 
Choose an integer $n \ge 1$ such that $nD$ is Cartier and linearly equivalent to $-nK_X$. 
It then suffices to show that $\mathrm{gfst}(X,D)\ge t$ for every rational number 
\[
    t \in (0,\min\{\alpha_F(Y),1\})\cap n\Z_{(p)}.
\]
Set $E:=\pi^*K_Y-X_X\ge 0$ and $D'=\pi_*D$. 
Then $D' \sim_\Q -K_Y$, so $D'$ is $\Q$-Cartier and we have $\pi^*D'=D-E$. 

Let $C$ be any effective Cartier divisor on $X$, and take an effective Cartier divisor $C'$ on $Y$ such that $\pi^*C' \ge C$. 
Since $t<\alpha_F(Y)$ and $t\in n\Z_{(p)}$, there exists an integer $e \ge 1$ such that $(p^e-1)t/n\in \Z$ and the map $\sO_Y\to F^e_*\sO_Y((p^e-1)tD'+C')$ splits as an $\sO_Y$-module homomorphism. 
Since
\begin{align*}
(1-p^e)K_X-(p^e-1)tD-\pi^*((1-p^e)K_Y-(p^e-1)tD')=(p^e-1)(1-t)E
\end{align*}
is an effective $\pi$-exceptional $\Q$-Weil divisor,
the splitting homomorphism $\phi$ lies in 
\begin{align*}
& \Hom_{\sO_Y}(F^e_*\sO_Y((p^e-1)tD'+C'),\sO_Y)\\ 
    &\cong H^0(Y,F^e_*\sO_Y((1-p^e)K_Y-(p^e-1)tD'-C'))\\
		& =  H^0(X,F^e_*\sO_X((1-p^e)K_X-(p^e-1)tD-\pi^*C'))\\
		&\subseteq  H^0(X,F^e_*\sO_X((1-p^e)K_X-(p^e-1)tD-C))\\
        &\cong \Hom_{\sO_X}(F^e_*\sO_X((p^e-1)tD+C),\sO_X).
	\end{align*}
Here we use the fact that for any Weil divisor $G$ on an $F$-finite normal variety $Z$, the adjunction formula (cf.~the proof of \cite[Theorem 3.3]{Hara-Watanabe}) gives 
\[
    {\mathcal{H}om}_{\sO_Z}(F^e_*\sO_Z(G),\sO_Z)\cong F^e_* \sO_Z((1-p^e)K_Z-G)). 
\]
Therefore, $\phi$ also gives a splitting of the map $\sO_X\to F^e_*\sO_X((p^e-1)tD+C)$, which completes the proof. 
\end{proof}

\begin{proof}[Proof of Theorem \ref{two conjectures are equivalent}]
It is clear that Conjecture $\mathrm{D}_{d}$ implies Conjecture $\mathrm{E}_d$, so we will prove the converse implication. 
Let $X$ be a $d$-dimensional complex projective variety of Fano type, that is, there exists an effective $\Q$-Weil divisor $\Delta$ on $X$ such that $(X, \Delta)$ is klt and $-(K_X+\Delta)$ is ample. 
By Lemma \ref{F-alpha invariant under small birational map}, after taking a $\Q$-factorization, we may assume that $X$ is $\Q$-factorial. 
Since $X$ is a Mori dream space by \cite{BCHM}, we can run a $(-K_X)$-minimal model program and obtain a sequence of divisorial contractions and flips: 
\[
\xymatrix{
X=X_0 \ar@{-->}[r]^{f_0} & X_1 \ar@{-->}[r]^{f_1} & \cdots \ar@{-->}[r]^{f_{l-2}} & X_{l-1} \ar@{-->}[r]^{f_{l-1}} & X_l=X',
}
\]
where each $X_i$ is of Fano type and $-K_{X'}$ is semi-ample. 
Then $-K_{X'}$ induces a birational contraction $g: X' \to Z$ such that $g^*K_Z=K_{X'}$ and $Z=\Proj R(X,-K_X)$ is the anti-canonical model of $X$. 

Now we reduce the situation to characteristic $p>0$. 
Suppose that 
\[
    (X_A, f_{j,A}: X_{j,A} \dashrightarrow X_{j+1,A}, Z_A)
\]
is a model of $(X,f_{j}: X_{j} \dashrightarrow X_{j+1}, Z)$ over a finitely generated $\Z$-subalgebra $A \subseteq \C$. 
If $f_{j}$ is a divisorial contraction, then $K_{X_{j}}$ is $f_j$-ample, so the Negativity Lemma implies that 
\[
    f_j^*K_{X_{j+1}}-K_{X_{j}}
\]    
is effective. 
Applying Lemma \ref{F-alpha invariant under divisorial contraction}, we then have $\alpha_F(X_{j,\mu})\ge \min\{\alpha_F(X_{j+1,\mu}),1\}$ for general closed points $\mu \in \Spec A$. 
If $f_{j}$ is a flip, then by Lemma \ref{F-alpha invariant under small birational map}, we have $\alpha_F(X_{j,\mu})=\alpha_F(X_{j+1,\mu})$ for general closed points $\mu \in \Spec A$.  
Therefore, 
\[
\alpha_F(X_{\mu})\ge \min\{\alpha_F(X'_{\mu}),1\} \ge \min\{\alpha_F(Z_{\mu}),1\}
\]
for general closed points $\mu \in \Spec A$, where the second inequality again follows from Lemma \ref{F-alpha invariant under divisorial contraction}. 
Finally, by applying Conjecture $\mathrm{E}_d$ to $Z$, we conclude that $\alpha_F(X)>0$. 

\end{proof}



Summing up, we have seen that the following implications hold among the conjectures discussed in this paper: 
\[
\xymatrix@C=50pt{
& \textup{Conjecture~$\mathrm{A}_{d+1}$} \ar@{=>}[d] & \\ 
\textup{Conjecture~$\mathrm{C}_d$} \ar@{=>}[r]^-{\textup{Theorem \ref{local-global thm}}} \ar@{=>}[d]^-{\textup{Proposition\ref{easy implications} (2)}} & \textup{Conjecture~$\mathrm{B}_{d+1}$} \ar@{=>}[r]^-{\textup{Corollary \ref{reduce to Gorenstein case}}} \ar@{=>}[d]^-{\textup{Proposition \ref{easy implications} (1)}} & \textup{Conjecture~$\mathrm{A}_d$}  \\
\textup{Conjecture~$\mathrm{D}_d$} \ar@{<=>}[r]^-{\textup{Theorem \ref{two conjectures are equivalent}}} & \textup{Conjecture~$\mathrm{E}_d$}.  
}
\]
(More precisely, the implication ``Condition $\mathrm{C}_d$ $\Rightarrow$ Conjecture $\mathrm{B}_{d+1}$'' holds when the ring has residue field isomorphic to $\C$.)

We conclude this section with the following question. 
\begin{ques}
Does Conjecture $\mathrm{D}_d$ imply Conjecture $\mathrm{C}_d$? 
If so, Conjecture $A$ would be equivalent to Conjecture~$E$ and could be reduced to the case where $R$ is a Gorenstein graded ring. 
\end{ques}

\begin{rem}
A key difficulty in the above question is that the global $F$-split region 
\[
\mathrm{cl}(\{(s,t) \in \R^2_{\ge 0} \; | \; (V, sB+tD) \; \textup{is globally $F$-regular}\}), 
\]
where $\mathrm{cl}(-)$ denotes closure in the Euclidean topology, $V$ is a globally $F$-regular variety and $B, D$ are effective $\Q$-Weil divisors on $V$, is not a polytope unlike log canonical regions. 
\end{rem}


\bibliography{TY.bib}

@article {CRST,
    AUTHOR = {Carvajal-Rojas, Javier and Schwede, Karl and Tucker, Kevin},
     TITLE = {Fundamental groups of {$F$}-regular singularities via
              {$F$}-signature},
   JOURNAL = {Ann. Sci. \'Ec. Norm. Sup\'er. (4)},
  FJOURNAL = {Annales Scientifiques de l'\'Ecole Normale Sup\'erieure.
              Quatri\`eme S\'erie},
    VOLUME = {51},
      YEAR = {2018},
    NUMBER = {4},
     PAGES = {993--1016},
      ISSN = {0012-9593,1873-2151},
   MRCLASS = {14B05 (14F35 14G17)},
  MRNUMBER = {3861567},
MRREVIEWER = {Daniel\ Smolkin},
       DOI = {10.24033/asens.2370},
       URL = {https://doi.org/10.24033/asens.2370},
}

@article {EGAIV3,
    AUTHOR = {Grothendieck, A.},
     TITLE = {\'{E}l\'{e}ments de g\'{e}om\'{e}trie alg\'{e}brique. {IV}. \'{E}tude locale des
              sch\'{e}mas et des morphismes de sch\'{e}mas. {III}},
   JOURNAL = {Inst. Hautes \'{E}tudes Sci. Publ. Math.},
  FJOURNAL = {Institut des Hautes \'{E}tudes Scientifiques. Publications
              Math\'{e}matiques},
    NUMBER = {28},
      YEAR = {1966},
     PAGES = {5--255},
      ISSN = {0073-8301},
   MRCLASS = {14.55},
  MRNUMBER = {217086},
MRREVIEWER = {J. P. Murre},
       URL = {http://www.numdam.org/item?id=PMIHES_1966__28__255_0},
}

@book {Kol13,
    AUTHOR = {Koll\'{a}r, J\'{a}nos},
     TITLE = {Singularities of the minimal model program},
    SERIES = {Cambridge Tracts in Mathematics},
    VOLUME = {200},
      NOTE = {With a collaboration of S\'{a}ndor Kov\'{a}cs},
 PUBLISHER = {Cambridge University Press, Cambridge},
      YEAR = {2013},
     PAGES = {x+370},
      ISBN = {978-1-107-03534-8},
   MRCLASS = {14E30 (14B05)},
  MRNUMBER = {3057950},
MRREVIEWER = {Tommaso De Fernex},
       DOI = {10.1017/CBO9781139547895},
       URL = {https://doi.org/10.1017/CBO9781139547895},
}

@article {SS10,
    AUTHOR = {Schwede, Karl and Smith, Karen E.},
     TITLE = {Globally {$F$}-regular and log {F}ano varieties},
   JOURNAL = {Adv. Math.},
  FJOURNAL = {Advances in Mathematics},
    VOLUME = {224},
      YEAR = {2010},
    NUMBER = {3},
     PAGES = {863--894},
      ISSN = {0001-8708},
   MRCLASS = {14J45 (13A35 14B05)},
  MRNUMBER = {2628797},
MRREVIEWER = {Yukihide Takayama},
       DOI = {10.1016/j.aim.2009.12.020},
       URL = {https://doi.org/10.1016/j.aim.2009.12.020},
}

@article {Hara-Watanabe,
    AUTHOR = {Hara, Nobuo and Watanabe, Kei-Ichi},
     TITLE = {F-regular and {F}-pure rings vs. log terminal and log
              canonical singularities},
   JOURNAL = {J. Algebraic Geom.},
  FJOURNAL = {Journal of Algebraic Geometry},
    VOLUME = {11},
      YEAR = {2002},
    NUMBER = {2},
     PAGES = {363--392},
      ISSN = {1056-3911},
   MRCLASS = {13A35 (14B05 14J17)},
  MRNUMBER = {1874118},
MRREVIEWER = {Karen E. Smith},
       DOI = {10.1090/S1056-3911-01-00306-X},
       URL = {https://doi.org/10.1090/S1056-3911-01-00306-X},
}

@article {Takagi,
    AUTHOR = {Takagi, Shunsuke},
     TITLE = {An interpretation of multiplier ideals via tight closure},
   JOURNAL = {J. Algebraic Geom.},
  FJOURNAL = {Journal of Algebraic Geometry},
    VOLUME = {13},
      YEAR = {2004},
    NUMBER = {2},
     PAGES = {393--415},
      ISSN = {1056-3911},
   MRCLASS = {13A35 (14B05 14E30)},
  MRNUMBER = {2047704},
MRREVIEWER = {Karen E. Smith},
       DOI = {10.1090/S1056-3911-03-00366-7},
       URL = {https://doi.org/10.1090/S1056-3911-03-00366-7},
}

@article {Satriano,
    AUTHOR = {Satriano, Matthew},
     TITLE = {The {C}hevalley-{S}hephard-{T}odd theorem for finite linearly
              reductive group schemes},
   JOURNAL = {Algebra Number Theory},
  FJOURNAL = {Algebra \& Number Theory},
    VOLUME = {6},
      YEAR = {2012},
    NUMBER = {1},
     PAGES = {1--26},
      ISSN = {1937-0652},
   MRCLASS = {13A50 (14L15)},
  MRNUMBER = {2950159},
MRREVIEWER = {Alan Koch},
       DOI = {10.2140/ant.2012.6.1},
       URL = {https://doi.org/10.2140/ant.2012.6.1},
}

@article {Takagi-Watanabe,
    AUTHOR = {Takagi, Shunsuke and Watanabe, Kei-Ichi},
     TITLE = {{$F$}-singularities: applications of characteristic {$p$}
              methods to singularity theory [translation of {MR}3135334]},
   JOURNAL = {Sugaku Expositions},
  FJOURNAL = {Sugaku Expositions},
    VOLUME = {31},
      YEAR = {2018},
    NUMBER = {1},
     PAGES = {1--42},
      ISSN = {0898-9583},
   MRCLASS = {13A35 (14B05)},
  MRNUMBER = {3784697},
MRREVIEWER = {Ehsan Tavanfar},
       DOI = {10.1090/suga/427},
       URL = {https://doi.org/10.1090/suga/427},
}

@misc{stacks-project,
    AUTHOR = "Stacks Project Authors, The",
    TITLE = "The {S}tacks {P}roject",
    URL = "https://stacks.math.columbia.edu/",
    howpublished = {\url{https://stacks.math.columbia.edu/}},
}

@article {Schoutens05,
    AUTHOR = {Schoutens, Hans},
     TITLE = {Log-terminal singularities and vanishing theorems via
              non-standard tight closure},
   JOURNAL = {J. Algebraic Geom.},
  FJOURNAL = {Journal of Algebraic Geometry},
    VOLUME = {14},
      YEAR = {2005},
    NUMBER = {2},
     PAGES = {357--390},
      ISSN = {1056-3911,1534-7486},
   MRCLASS = {13A35 (14E05)},
  MRNUMBER = {2123234},
MRREVIEWER = {Karen\ E.\ Smith},
       DOI = {10.1090/S1056-3911-04-00395-9},
       URL = {https://doi.org/10.1090/S1056-3911-04-00395-9},
}

@incollection {MTW05,
    AUTHOR = {Musta\c{t}\u{a}, Mircea and Takagi, Shunsuke and Watanabe,
              Kei-ichi},
     TITLE = {F-thresholds and {B}ernstein-{S}ato polynomials},
 BOOKTITLE = {European {C}ongress of {M}athematics},
     PAGES = {341--364},
 PUBLISHER = {Eur. Math. Soc., Z\"urich},
      YEAR = {2005},
      ISBN = {3-03719-009-4},
   MRCLASS = {13A35 (14B05)},
  MRNUMBER = {2185754},
MRREVIEWER = {Karen\ E.\ Smith},
}

@article {TW04,
    AUTHOR = {Takagi, Shunsuke and Watanabe, Kei-ichi},
     TITLE = {On ${F}$-pure thresholds},
   JOURNAL = {J. Algebra},
  FJOURNAL = {Journal of Algebra},
    VOLUME = {282},
      YEAR = {2004},
    NUMBER = {1},
     PAGES = {278--297},
      ISSN = {0021-8693,1090-266X},
   MRCLASS = {13A35},
  MRNUMBER = {2097584},
MRREVIEWER = {Oleg\ N.\ Popov},
       DOI = {10.1016/j.jalgebra.2004.07.011},
       URL = {https://doi.org/10.1016/j.jalgebra.2004.07.011},
}

@article {Huebl01,
    AUTHOR = {H{\"u}bl, Reinhold},
     TITLE = {Completions of local morphisms and valuations},
   JOURNAL = {Math. Z.},
  FJOURNAL = {Mathematische Zeitschrift},
    VOLUME = {236},
      YEAR = {2001},
    NUMBER = {1},
     PAGES = {201--214},
      ISSN = {0025-5874,1432-1823},
   MRCLASS = {13A18 (13B22 13B35 13B40 13J07 13N05)},
  MRNUMBER = {1812456},
MRREVIEWER = {Irena\ Swanson},
       DOI = {10.1007/PL00004824},
       URL = {https://doi.org/10.1007/PL00004824},
}

@article {Yamaguchi25a,
    AUTHOR = {Yamaguchi, Tatsuki},
     TITLE = {{$F$}-pure and {$F$}-injective singularities in equal
              characteristic zero},
   JOURNAL = {Math. Res. Lett.},
  FJOURNAL = {Mathematical Research Letters},
    VOLUME = {32},
      YEAR = {2025},
    NUMBER = {5},
     PAGES = {1667--1695},
      ISSN = {1073-2780,1945-001X},
   MRCLASS = {13A35},
  MRNUMBER = {5009222},
       DOI = {10.4310/mrl.251215233345},
       URL = {https://doi.org/10.4310/mrl.251215233345},
}

@incollection {AE08,
    AUTHOR = {Aberbach, Ian M. and Enescu, Florian},
     TITLE = {Lower bounds for {H}ilbert-{K}unz multiplicities in local
              rings of fixed dimension},
      NOTE = {Special volume in honor of Melvin Hochster},
      booktitle = {Special Volume in Honor of Melvin Hochster},
   JOURNAL = {Michigan Math. J.},
  FJOURNAL = {Michigan Mathematical Journal},
    VOLUME = {57},
      YEAR = {2008},
     PAGES = {1--16},
      ISSN = {0026-2285,1945-2365},
   MRCLASS = {13D40 (13H15)},
  MRNUMBER = {2492437},
       DOI = {10.1307/mmj/1220879393},
       URL = {https://doi.org/10.1307/mmj/1220879393},
  publisher = {Department of Mathematics, University of Michigan},
}

@article {Tucker12,
    AUTHOR = {Tucker, Kevin},
     TITLE = {{$F$}-signature exists},
   JOURNAL = {Invent. Math.},
  FJOURNAL = {Inventiones Mathematicae},
    VOLUME = {190},
      YEAR = {2012},
    NUMBER = {3},
     PAGES = {743--765},
      ISSN = {0020-9910,1432-1297},
   MRCLASS = {13A35 (13D40 14B05)},
  MRNUMBER = {2995185},
MRREVIEWER = {Alberto\ F.\ Boix},
       DOI = {10.1007/s00222-012-0389-0},
       URL = {https://doi.org/10.1007/s00222-012-0389-0},
}

@article {HL02,
    AUTHOR = {Huneke, Craig and Leuschke, Graham J.},
     TITLE = {Two theorems about maximal {C}ohen-{M}acaulay modules},
   JOURNAL = {Math. Ann.},
  FJOURNAL = {Mathematische Annalen},
    VOLUME = {324},
      YEAR = {2002},
    NUMBER = {2},
     PAGES = {391--404},
      ISSN = {0025-5831,1432-1807},
   MRCLASS = {13C14 (13A35 13D07 13H10)},
  MRNUMBER = {1933863},
MRREVIEWER = {\=I.\ \=I.\ Burban},
       DOI = {10.1007/s00208-002-0343-3},
       URL = {https://doi.org/10.1007/s00208-002-0343-3},
}

@article {Schoutens04,
    AUTHOR = {Schoutens, Hans},
     TITLE = {Canonical big {C}ohen-{M}acaulay algebras and rational
              singularities},
   JOURNAL = {Illinois J. Math.},
  FJOURNAL = {Illinois Journal of Mathematics},
    VOLUME = {48},
      YEAR = {2004},
    NUMBER = {1},
     PAGES = {131--150},
      ISSN = {0019-2082,1945-6581},
   MRCLASS = {13C14 (13A35 14B05)},
  MRNUMBER = {2048219},
MRREVIEWER = {D.-M.\ Popescu},
       URL = {http://projecteuclid.org/euclid.ijm/1258136178},
}

@article {Yao06,
    AUTHOR = {Yao, Yongwei},
     TITLE = {Observations on the {$F$}-signature of local rings of
              characteristic {$p$}},
   JOURNAL = {J. Algebra},
  FJOURNAL = {Journal of Algebra},
    VOLUME = {299},
      YEAR = {2006},
    NUMBER = {1},
     PAGES = {198--218},
      ISSN = {0021-8693,1090-266X},
   MRCLASS = {13A35 (13H10)},
  MRNUMBER = {2225772},
MRREVIEWER = {Ian\ M.\ Aberbach},
       DOI = {10.1016/j.jalgebra.2005.08.013},
       URL = {https://doi.org/10.1016/j.jalgebra.2005.08.013},
}

@misc{Taylor20,
      title={Inversion of adjunction for {$F$}-signature}, 
      author={Gregory Taylor},
      year={2020},
      eprint={1909.10436},
      archivePrefix={arXiv},
      primaryClass={math.AC},
      url={https://arxiv.org/abs/1909.10436},
      note={arXiv:1909.10436}, 
}

@misc{Pande23,
      title={A {F}robenius version of {T}ian's alpha-invariant}, 
      author={Suchitra Pande},
      year={2023},
      eprint={2311.00989},
      archivePrefix={arXiv},
      primaryClass={math.AG},
      url={https://arxiv.org/abs/2311.00989}, 
      note={arXiv:2311.00989}, 
}

@article {GOST15,
    AUTHOR = {Gongyo, Yoshinori and Okawa, Shinnosuke and Sannai, Akiyoshi
              and Takagi, Shunsuke},
     TITLE = {Characterization of varieties of {F}ano type via singularities
              of {C}ox rings},
   JOURNAL = {J. Algebraic Geom.},
  FJOURNAL = {Journal of Algebraic Geometry},
    VOLUME = {24},
      YEAR = {2015},
    NUMBER = {1},
     PAGES = {159--182},
      ISSN = {1056-3911,1534-7486},
   MRCLASS = {14E30 (14C20)},
  MRNUMBER = {3275656},
MRREVIEWER = {Paul\ A.\ Hacking},
       DOI = {10.1090/S1056-3911-2014-00641-X},
       URL = {https://doi.org/10.1090/S1056-3911-2014-00641-X},
}

@article {SB18,
    AUTHOR = {De Stefani, Alessandro and N\'u\~nez-Betancourt, Luis},
     TITLE = {{$F$}-thresholds of graded rings},
   JOURNAL = {Nagoya Math. J.},
  FJOURNAL = {Nagoya Mathematical Journal},
    VOLUME = {229},
      YEAR = {2018},
     PAGES = {141--168},
      ISSN = {0027-7630,2152-6842},
   MRCLASS = {13A35 (13H10 14B05)},
  MRNUMBER = {3778235},
MRREVIEWER = {Thomas\ Polstra},
       DOI = {10.1017/nmj.2016.65},
       URL = {https://doi.org/10.1017/nmj.2016.65},
}

@article {Polstra26,
    AUTHOR = {Polstra, Thomas},
     TITLE = {Zariski--{N}agata theorems for singularities and the {U}niform
              {I}zumi--{R}ees {P}roperty},
   JOURNAL = {Compos. Math.},
  FJOURNAL = {Compositio Mathematica},
    VOLUME = {162},
      YEAR = {2026},
    NUMBER = {3},
     PAGES = {602--629},
      ISSN = {0010-437X,1570-5846},
   MRCLASS = {13H15 (13A18 13A30 14C17)},
  MRNUMBER = {5089840},
       DOI = {10.1017/S0010437X26103133},
       URL = {https://doi.org/10.1017/S0010437X26103133},
}

@article {Watanabe81,
    AUTHOR = {Watanabe, Keiichi},
     TITLE = {Some remarks concerning {D}emazure's construction of normal
              graded rings},
   JOURNAL = {Nagoya Math. J.},
  FJOURNAL = {Nagoya Mathematical Journal},
    VOLUME = {83},
      YEAR = {1981},
     PAGES = {203--211},
      ISSN = {0027-7630,2152-6842},
   MRCLASS = {13H10 (14B05)},
  MRNUMBER = {632654},
MRREVIEWER = {Sarah\ Glaz},
       URL = {http://projecteuclid.org/euclid.nmj/1118786485},
}

@article {HLS26,
    AUTHOR = {Han, Jingjun and Liu, Jihao and Shokurov, V. V.},
     TITLE = {A{CC} for {M}inimal {L}og {D}iscrepancies of {E}xceptional
              {S}ingularities},
   JOURNAL = {Peking Math. J.},
  FJOURNAL = {Peking Mathematical Journal},
    VOLUME = {9},
      YEAR = {2026},
    NUMBER = {2},
     PAGES = {225--257},
      ISSN = {2096-6075,2524-7182},
   MRCLASS = {14E30},
  MRNUMBER = {5074546},
       DOI = {10.1007/s42543-024-00091-x},
       URL = {https://doi.org/10.1007/s42543-024-00091-x},
}

@book {Kollaretal,
  author = {J. Koll\'ar and others},
     TITLE = {Flips and abundance for algebraic threefolds},
      NOTE = {Papers from the Second Summer Seminar on Algebraic Geometry
              held at the University of Utah, Salt Lake City, Utah, August
              1991,
              Ast\'erisque No. 211 (1992)},
 PUBLISHER = {Soci\'et\'e{} Math\'ematique de France, Paris},
      YEAR = {1992},
     PAGES = {1--258},
      ISSN = {0303-1179,2492-5926},
   MRCLASS = {14E30 (14E35 14M10)},
  MRNUMBER = {1225842},
MRREVIEWER = {Mark\ Gross},
}

@article {Schwede-Tucker14,
    AUTHOR = {Schwede, Karl and Tucker, Kevin},
     TITLE = {On the behavior of test ideals under finite morphisms},
   JOURNAL = {J. Algebraic Geom.},
  FJOURNAL = {Journal of Algebraic Geometry},
    VOLUME = {23},
      YEAR = {2014},
    NUMBER = {3},
     PAGES = {399--443},
      ISSN = {1056-3911,1534-7486},
   MRCLASS = {14F18 (14G17)},
  MRNUMBER = {3205587},
MRREVIEWER = {Tomasz\ Szemberg},
       DOI = {10.1090/S1056-3911-2013-00610-4},
       URL = {https://doi.org/10.1090/S1056-3911-2013-00610-4},
}

@article {Yam23a,
    AUTHOR = {Yamaguchi, Tatsuki},
     TITLE = {A characterization of multiplier ideals via ultraproducts},
   JOURNAL = {Manuscripta Math.},
  FJOURNAL = {Manuscripta Mathematica},
    VOLUME = {172},
      YEAR = {2023},
    NUMBER = {3-4},
     PAGES = {1153--1168},
      ISSN = {0025-2611,1432-1785},
   MRCLASS = {13A35 (14B05 14F18)},
  MRNUMBER = {4651117},
       DOI = {10.1007/s00229-022-01446-3},
       URL = {https://doi.org/10.1007/s00229-022-01446-3},
}

@misc{Yamaguchi25,
      title={Ultra-test ideals in rings with finitely generated anti-canonical algebras}, 
      author={Tatsuki Yamaguchi},
      note={arXiv:2502.03816}, 
      year={2025},
      eprint={2502.03816},
      archivePrefix={arXiv},
      primaryClass={math.AG},
      url={https://arxiv.org/abs/2502.03816}, 
}

@incollection {Schwede-Takagi08,
    AUTHOR = {Schwede, Karl and Takagi, Shunsuke},
     TITLE = {Rational singularities associated to pairs},
      NOTE = {Special volume in honor of Melvin Hochster},
booktitle = {Special Volume in Honor of Melvin Hochster},
   JOURNAL = {Michigan Math. J.},
  FJOURNAL = {Michigan Mathematical Journal},
    VOLUME = {57},
      YEAR = {2008},
     PAGES = {625--658},
      ISSN = {0026-2285,1945-2365},
   MRCLASS = {14B05 (14E30)},
  MRNUMBER = {2492473},
MRREVIEWER = {Marko\ Roczen},
       DOI = {10.1307/mmj/1220879429},
       URL = {https://doi.org/10.1307/mmj/1220879429},
 publisher = {Department of Mathematics, University of Michigan},
}

@incollection {Watanabe94,
    AUTHOR = {Watanabe, Keiichi},
     TITLE = {Infinite cyclic covers of strongly {$F$}-regular rings},
 BOOKTITLE = {Commutative algebra: syzygies, multiplicities, and birational
              algebra ({S}outh {H}adley, {MA}, 1992)},
    SERIES = {Contemp. Math.},
    VOLUME = {159},
     PAGES = {423--432},
 PUBLISHER = {Amer. Math. Soc., Providence, RI},
      YEAR = {1994},
      ISBN = {0-8218-5188-8},
   MRCLASS = {13H10 (13D45 14B05)},
  MRNUMBER = {1266196},
MRREVIEWER = {Manfred\ Herrmann},
       DOI = {10.1090/conm/159/01521},
       URL = {https://doi.org/10.1090/conm/159/01521},
}

@article {BCHM,
    AUTHOR = {Birkar, Caucher and Cascini, Paolo and Hacon, Christopher D.
              and McKernan, James},
     TITLE = {Existence of minimal models for varieties of log general type},
   JOURNAL = {J. Amer. Math. Soc.},
  FJOURNAL = {Journal of the American Mathematical Society},
    VOLUME = {23},
      YEAR = {2010},
    NUMBER = {2},
     PAGES = {405--468},
      ISSN = {0894-0347,1088-6834},
   MRCLASS = {14E30 (14E05)},
  MRNUMBER = {2601039},
MRREVIEWER = {Mark\ Gross},
       DOI = {10.1090/S0894-0347-09-00649-3},
       URL = {https://doi.org/10.1090/S0894-0347-09-00649-3},
}

@incollection {vdD79,
    AUTHOR = {van den Dries, Lou},
     TITLE = {Algorithms and bounds for polynomial rings},
 BOOKTITLE = {Logic {C}olloquium '78 ({M}ons, 1978)},
    SERIES = {Stud. Logic Found. Math.},
    VOLUME = {97},
     PAGES = {147--157},
 PUBLISHER = {North-Holland, Amsterdam-New York},
      YEAR = {1979},
      ISBN = {0-444-85378-2},
   MRCLASS = {03C60 (12E05)},
  MRNUMBER = {567669},
MRREVIEWER = {Annalisa\ Marcja},
}

@article {Schoutens03,
    AUTHOR = {Schoutens, Hans},
     TITLE = {Non-standard tight closure for affine {$\mathbb{C}$}-algebras},
   JOURNAL = {Manuscripta Math.},
  FJOURNAL = {Manuscripta Mathematica},
    VOLUME = {111},
      YEAR = {2003},
    NUMBER = {3},
     PAGES = {379--412},
      ISSN = {0025-2611,1432-1785},
   MRCLASS = {13A35 (03C20)},
  MRNUMBER = {1993501},
MRREVIEWER = {Holger\ Brenner},
       DOI = {10.1007/s00229-003-0380-6},
       URL = {https://doi.org/10.1007/s00229-003-0380-6},
}

@article {Schoutens08,
    AUTHOR = {Schoutens, Hans},
     TITLE = {Pure subrings of regular rings are pseudo-rational},
   JOURNAL = {Trans. Amer. Math. Soc.},
  FJOURNAL = {Transactions of the American Mathematical Society},
    VOLUME = {360},
      YEAR = {2008},
    NUMBER = {2},
     PAGES = {609--627},
      ISSN = {0002-9947,1088-6850},
   MRCLASS = {13A35 (14B05)},
  MRNUMBER = {2346464},
MRREVIEWER = {Karl\ Schwede},
       DOI = {10.1090/S0002-9947-07-04134-7},
       URL = {https://doi.org/10.1090/S0002-9947-07-04134-7},
}

@article {Hara-Takagi,
    AUTHOR = {Hara, Nobuo and Takagi, Shunsuke},
     TITLE = {On a generalization of test ideals},
   JOURNAL = {Nagoya Math. J.},
  FJOURNAL = {Nagoya Mathematical Journal},
    VOLUME = {175},
      YEAR = {2004},
     PAGES = {59--74},
      ISSN = {0027-7630,2152-6842},
   MRCLASS = {13A35},
  MRNUMBER = {2085311},
MRREVIEWER = {Ana\ Bravo},
       DOI = {10.1017/S0027763000008904},
       URL = {https://doi.org/10.1017/S0027763000008904},
}

@article {Li--Xu,
    AUTHOR = {Li, Chi and Xu, Chenyang},
     TITLE = {Stability of valuations and {K}oll\'ar components},
   JOURNAL = {J. Eur. Math. Soc. (JEMS)},
  FJOURNAL = {Journal of the European Mathematical Society (JEMS)},
    VOLUME = {22},
      YEAR = {2020},
    NUMBER = {8},
     PAGES = {2573--2627},
      ISSN = {1435-9855,1435-9863},
   MRCLASS = {14J17 (13A18 14E20)},
  MRNUMBER = {4118616},
MRREVIEWER = {Jos\'e\ Mar\'ia\ Tornero},
       DOI = {10.4171/JEMS/972},
       URL = {https://doi.org/10.4171/JEMS/972},
}

@article {Liu--Zhuang,
    AUTHOR = {Liu, Yuchen and Zhuang, Ziquan},
     TITLE = {On the sharpness of {T}ian's criterion for {K}-stability},
   JOURNAL = {Nagoya Math. J.},
  FJOURNAL = {Nagoya Mathematical Journal},
    VOLUME = {245},
      YEAR = {2022},
     PAGES = {41--73},
      ISSN = {0027-7630,2152-6842},
   MRCLASS = {14J45 (14L24 32Q26)},
  MRNUMBER = {4413362},
MRREVIEWER = {Guolei\ Zhong},
       DOI = {10.1017/nmj.2020.28},
       URL = {https://doi.org/10.1017/nmj.2020.28},
}

@article {Hara-Yoshida,
    AUTHOR = {Hara, Nobuo and Yoshida, Ken-Ichi},
     TITLE = {A generalization of tight closure and multiplier ideals},
   JOURNAL = {Trans. Amer. Math. Soc.},
  FJOURNAL = {Transactions of the American Mathematical Society},
    VOLUME = {355},
      YEAR = {2003},
    NUMBER = {8},
     PAGES = {3143--3174},
      ISSN = {0002-9947,1088-6850},
   MRCLASS = {13A35},
  MRNUMBER = {1974679},
MRREVIEWER = {Irena\ Swanson},
       DOI = {10.1090/S0002-9947-03-03285-9},
       URL = {https://doi.org/10.1090/S0002-9947-03-03285-9},
}

@article {Hochster-Huneke90,
    AUTHOR = {Hochster, Melvin and Huneke, Craig},
     TITLE = {Tight closure, invariant theory, and the {B}rian\c con-{S}koda
              theorem.},
   JOURNAL = {J. Amer. Math. Soc.},
  FJOURNAL = {Journal of the American Mathematical Society},
    VOLUME = {3},
      YEAR = {1990},
    NUMBER = {1},
     PAGES = {31--116},
      ISSN = {0894-0347,1088-6834},
   MRCLASS = {13C05 (13A15 13A50 13B99 13D02)},
  MRNUMBER = {1017784},
MRREVIEWER = {Luchezar\ L.\ Avramov},
       DOI = {10.2307/1990984},
       URL = {https://doi.org/10.2307/1990984},
}

@article {BSTZ,
    AUTHOR = {Blickle, Manuel and Schwede, Karl and Takagi, Shunsuke and
              Zhang, Wenliang},
     TITLE = {Discreteness and rationality of {$F$}-jumping numbers on
              singular varieties},
   JOURNAL = {Math. Ann.},
  FJOURNAL = {Mathematische Annalen},
    VOLUME = {347},
      YEAR = {2010},
    NUMBER = {4},
     PAGES = {917--949},
      ISSN = {0025-5831,1432-1807},
   MRCLASS = {13A35 (14B05 14F18)},
  MRNUMBER = {2658149},
MRREVIEWER = {Ana\ Bravo},
       DOI = {10.1007/s00208-009-0461-2},
       URL = {https://doi.org/10.1007/s00208-009-0461-2},
}

@article {AL03,
    AUTHOR = {Aberbach, Ian M. and Leuschke, Graham J.},
     TITLE = {The {$F$}-signature and strong {$F$}-regularity},
   JOURNAL = {Math. Res. Lett.},
  FJOURNAL = {Mathematical Research Letters},
    VOLUME = {10},
      YEAR = {2003},
    NUMBER = {1},
     PAGES = {51--56},
      ISSN = {1073-2780},
   MRCLASS = {13A35},
  MRNUMBER = {1960123},
MRREVIEWER = {Karen\ E.\ Smith},
       DOI = {10.4310/MRL.2003.v10.n1.a6},
       URL = {https://doi.org/10.4310/MRL.2003.v10.n1.a6},
}

@article {TakTak08,
    AUTHOR = {Takagi, Shunsuke and Takahashi, Ryo},
     TITLE = {{$D$}-modules over rings with finite {$F$}-representation
              type},
   JOURNAL = {Math. Res. Lett.},
  FJOURNAL = {Mathematical Research Letters},
    VOLUME = {15},
      YEAR = {2008},
    NUMBER = {3},
     PAGES = {563--581},
      ISSN = {1073-2780},
   MRCLASS = {13A35 (13D45 13N10)},
  MRNUMBER = {2407232},
MRREVIEWER = {Manuel\ Blickle},
       DOI = {10.4310/MRL.2008.v15.n3.a15},
       URL = {https://doi.org/10.4310/MRL.2008.v15.n3.a15},
}

@incollection {BFJ14,
    AUTHOR = {Boucksom, S\'ebastien and Favre, Charles and Jonsson, Mattias},
     TITLE = {A refinement of {I}zumi's theorem},
 BOOKTITLE = {Valuation theory in interaction},
    SERIES = {EMS Ser. Congr. Rep.},
     PAGES = {55--81},
 PUBLISHER = {Eur. Math. Soc., Z\"urich},
      YEAR = {2014},
      ISBN = {978-3-03719-149-1},
   MRCLASS = {13A18 (14C20 14F18 32S05)},
  MRNUMBER = {3329027},
MRREVIEWER = {A.\ R.\ Wadsworth},
}

@book {Huneke-Swanson,
    AUTHOR = {Huneke, Craig and Swanson, Irena},
     TITLE = {Integral closure of ideals, rings, and modules},
    SERIES = {London Mathematical Society Lecture Note Series},
    VOLUME = {336},
 PUBLISHER = {Cambridge University Press, Cambridge},
      YEAR = {2006},
     PAGES = {xiv+431},
      ISBN = {978-0-521-68860-4; 0-521-68860-4},
   MRCLASS = {13B22 (13A18 13A30 13A35 13H15 14A05)},
  MRNUMBER = {2266432},
MRREVIEWER = {Liam\ O'Carroll},
}

@article {Zhuang24,
    AUTHOR = {Zhuang, Ziquan},
     TITLE = {Direct summands of klt singularities},
   JOURNAL = {Invent. Math.},
  FJOURNAL = {Inventiones Mathematicae},
    VOLUME = {237},
      YEAR = {2024},
    NUMBER = {3},
     PAGES = {1683--1695},
      ISSN = {0020-9910,1432-1297},
   MRCLASS = {14E30},
  MRNUMBER = {4777095},
       DOI = {10.1007/s00222-024-01281-1},
       URL = {https://doi.org/10.1007/s00222-024-01281-1},
}

@article {Takagi_invent,
    AUTHOR = {Takagi, Shunsuke},
     TITLE = {F-singularities of pairs and inversion of adjunction of
              arbitrary codimension},
   JOURNAL = {Invent. Math.},
  FJOURNAL = {Inventiones Mathematicae},
    VOLUME = {157},
      YEAR = {2004},
    NUMBER = {1},
     PAGES = {123--146},
      ISSN = {0020-9910,1432-1297},
   MRCLASS = {14E30 (13A35 14N30)},
  MRNUMBER = {2135186},
MRREVIEWER = {Karen\ E.\ Smith},
       DOI = {10.1007/s00222-003-0350-3},
       URL = {https://doi.org/10.1007/s00222-003-0350-3},
}

@misc {Odaka-Hattori,
      title={Minimization of {A}rakelov {K}-energy for many cases}, 
      author={Masafumi Hattori and Yuji Odaka},
      year={2024},
      eprint={2211.03415},
      archivePrefix={arXiv},
      primaryClass={math.AG},
      url={https://arxiv.org/abs/2211.03415}, 
      note={arXiv:2211.03415}, 
}

@article {Smith97,
    AUTHOR = {Smith, Karen E.},
     TITLE = {{$F$}-rational rings have rational singularities},
   JOURNAL = {Amer. J. Math.},
  FJOURNAL = {American Journal of Mathematics},
    VOLUME = {119},
      YEAR = {1997},
    NUMBER = {1},
     PAGES = {159--180},
      ISSN = {0002-9327,1080-6377},
   MRCLASS = {13A35 (13D45 13F40 14B05)},
  MRNUMBER = {1428062},
MRREVIEWER = {Ian\ M.\ Aberbach},
       URL =
              {http://muse.jhu.edu/journals/american_journal_of_mathematics/v119/119.1smith.pdf},
}

@article {CEMS18,
    AUTHOR = {Chiecchio, Alberto and Enescu, Florian and Miller, Lance
              Edward and Schwede, Karl},
     TITLE = {Test ideals in rings with finitely generated anti-canonical
              algebras},
   JOURNAL = {J. Inst. Math. Jussieu},
  FJOURNAL = {Journal of the Institute of Mathematics of Jussieu. JIMJ.
              Journal de l'Institut de Math\'ematiques de Jussieu},
    VOLUME = {17},
      YEAR = {2018},
    NUMBER = {1},
     PAGES = {171--206},
      ISSN = {1474-7480,1475-3030},
   MRCLASS = {13A35 (14B05 14J17)},
  MRNUMBER = {3742559},
MRREVIEWER = {Alberto\ F.\ Boix},
       DOI = {10.1017/S1474748015000456},
       URL = {https://doi.org/10.1017/S1474748015000456},
}

@article {Singh05,
    AUTHOR = {Singh, Anurag K.},
     TITLE = {The {$F$}-signature of an affine semigroup ring},
   JOURNAL = {J. Pure Appl. Algebra},
  FJOURNAL = {Journal of Pure and Applied Algebra},
    VOLUME = {196},
      YEAR = {2005},
    NUMBER = {2-3},
     PAGES = {313--321},
      ISSN = {0022-4049,1873-1376},
   MRCLASS = {13A35},
  MRNUMBER = {2110527},
MRREVIEWER = {Ian\ M.\ Aberbach},
       DOI = {10.1016/j.jpaa.2004.08.001},
       URL = {https://doi.org/10.1016/j.jpaa.2004.08.001},
}

@misc{CSTZ2024,
      title={F-signature functions of diagonal hypersurfaces}, 
      author={Alessio Caminata and Samuel Shideler and Kevin Tucker and Francesco Zerman},
      year={2024},
      eprint={2403.12863},
      archivePrefix={arXiv},
      primaryClass={math.AC},
      url={https://arxiv.org/abs/2403.12863}, 
      note={arXiv:2403.12863}, 
}

@misc{BCPT2025,
      title={Limit {$F$}-signature functions of two-variable binomial hypersurfaces}, 
      author={Anna Brosowsky and Izzet Coskun and Suchitra Pande and Kevin Tucker},
      year={2025},
      eprint={2504.18656},
      archivePrefix={arXiv},
      primaryClass={math.AC},
      url={https://arxiv.org/abs/2504.18656}, 
      note={arXiv:2504.18656}, 
}

@article{HH_char0,
      title={Tight closure in equal characteristic zero}, 
      author={Melvin Hochster and Craig Huneke},
      journal={Preprint},
      year={2020}, 
}

@article {Smith-vandenBergh,
    AUTHOR = {Smith, Karen E. and Van den Bergh, Michel},
     TITLE = {Simplicity of rings of differential operators in prime
              characteristic},
   JOURNAL = {Proc. London Math. Soc. (3)},
  FJOURNAL = {Proceedings of the London Mathematical Society. Third Series},
    VOLUME = {75},
      YEAR = {1997},
    NUMBER = {1},
     PAGES = {32--62},
      ISSN = {0024-6115,1460-244X},
   MRCLASS = {16S32},
  MRNUMBER = {1444312},
MRREVIEWER = {Ian\ M.\ Musson},
       DOI = {10.1112/S0024611597000257},
       URL = {https://doi.org/10.1112/S0024611597000257},
}

@article {GHNV90,
    AUTHOR = {Goto, S. and Herrmann, M. and Nishida, K. and Villamayor, O.},
     TITLE = {On the structure of {N}oetherian symbolic {R}ees algebras},
   JOURNAL = {Manuscripta Math.},
  FJOURNAL = {Manuscripta Mathematica},
    VOLUME = {67},
      YEAR = {1990},
    NUMBER = {2},
     PAGES = {197--225},
      ISSN = {0025-2611,1432-1785},
   MRCLASS = {13C05 (13C15 13E05 13H10)},
  MRNUMBER = {1042238},
MRREVIEWER = {Aron\ Simis},
       DOI = {10.1007/BF02568430},
       URL = {https://doi.org/10.1007/BF02568430},
}

@article {Schwede08,
    AUTHOR = {Schwede, Karl},
     TITLE = {Generalized test ideals, sharp {$F$}-purity, and sharp test
              elements},
   JOURNAL = {Math. Res. Lett.},
  FJOURNAL = {Mathematical Research Letters},
    VOLUME = {15},
      YEAR = {2008},
    NUMBER = {6},
     PAGES = {1251--1261},
      ISSN = {1073-2780},
   MRCLASS = {13A35 (14B05)},
  MRNUMBER = {2470398},
       DOI = {10.4310/MRL.2008.v15.n6.a14},
       URL = {https://doi.org/10.4310/MRL.2008.v15.n6.a14},
}

@article {MS11,
    AUTHOR = {Musta\c{t}\u{a}, Mircea and Srinivas, Vasudevan},
     TITLE = {Ordinary varieties and the comparison between multiplier
              ideals and test ideals},
   JOURNAL = {Nagoya Math. J.},
  FJOURNAL = {Nagoya Mathematical Journal},
    VOLUME = {204},
      YEAR = {2011},
     PAGES = {125--157},
      ISSN = {0027-7630,2152-6842},
   MRCLASS = {14F18 (13A35)},
  MRNUMBER = {2863367},
MRREVIEWER = {Shin-Yao\ Jow},
       DOI = {10.1215/00277630-1431849},
       URL = {https://doi.org/10.1215/00277630-1431849},
}

@article {ELS03,
    AUTHOR = {Ein, Lawrence and Lazarsfeld, Robert and Smith, Karen E.},
     TITLE = {Uniform approximation of {A}bhyankar valuation ideals in
              smooth function fields},
   JOURNAL = {Amer. J. Math.},
  FJOURNAL = {American Journal of Mathematics},
    VOLUME = {125},
      YEAR = {2003},
    NUMBER = {2},
     PAGES = {409--440},
      ISSN = {0002-9327,1080-6377},
   MRCLASS = {13A18 (13H10)},
  MRNUMBER = {1963690},
MRREVIEWER = {Irena\ Swanson},
       URL =
              {http://muse.jhu.edu/journals/american_journal_of_mathematics/v125/125.2ein.pdf},
}

@article {Cutkosky13,
    AUTHOR = {Cutkosky, Steven Dale},
     TITLE = {Multiplicities associated to graded families of ideals},
   JOURNAL = {Algebra Number Theory},
  FJOURNAL = {Algebra \& Number Theory},
    VOLUME = {7},
      YEAR = {2013},
    NUMBER = {9},
     PAGES = {2059--2083},
      ISSN = {1937-0652,1944-7833},
   MRCLASS = {13H15 (13H05 14B05 14C20)},
  MRNUMBER = {3152008},
MRREVIEWER = {Catalin\ Ciuperca},
       DOI = {10.2140/ant.2013.7.2059},
       URL = {https://doi.org/10.2140/ant.2013.7.2059},
}

@article {Hubl-Swanson,
    AUTHOR = {H\"ubl, Reinhold and Swanson, Irena},
     TITLE = {Discrete valuations centered on local domains},
   JOURNAL = {J. Pure Appl. Algebra},
  FJOURNAL = {Journal of Pure and Applied Algebra},
    VOLUME = {161},
      YEAR = {2001},
    NUMBER = {1-2},
     PAGES = {145--166},
      ISSN = {0022-4049,1873-1376},
   MRCLASS = {13A18 (13A30 13F40)},
  MRNUMBER = {1834082},
MRREVIEWER = {Tiberiu\ Dumitrescu},
       DOI = {10.1016/S0022-4049(00)00086-4},
       URL = {https://doi.org/10.1016/S0022-4049(00)00086-4},
}
\bibliographystyle{alpha}

\bigskip

\end{document}